\documentclass[preprint,12pt]{elsarticle}
\usepackage{amsmath,amssymb,amsthm}
\usepackage{booktabs}
\usepackage{graphicx}
\usepackage[expansion=false]{microtype}
\usepackage{url}
\journal{Stochastic Processes and their Applications}
\numberwithin{equation}{section}
\theoremstyle{plain}
\newtheorem{theorem}{Theorem}[section]
\newtheorem{proposition}[theorem]{Proposition}
\newtheorem{corollary}[theorem]{Corollary}
\newtheorem{lemma}[theorem]{Lemma}
\theoremstyle{definition}

\newtheorem{remark}[theorem]{Remark}
\newcommand{\E}{\mathbb{E}}
\newcommand{\PP}{\mathbb{P}}
\newcommand{\dd}{\mathrm{d}}
\newcommand{\KL}{\mathrm{KL}}
\newcommand{\TV}{\mathrm{TV}}

\newcommand{\parhead}[1]{\par\medskip\noindent\textbf{#1}\enspace}
\begin{document}
\begin{frontmatter}
\title{Learning the exogenous rate but not the distance to criticality in nearly unstable heavy-tailed Hawkes processes}
\author[udd]{Mauricio Herrera-Mar\'{i}n\corref{cor1}}
\ead{mherrera@udd.cl}
\cortext[cor1]{Corresponding author.}
\address[udd]{Faculty of Engineering, Universidad del Desarrollo, Avda. Plaza 680, Las Condes, Santiago, Chile}
\begin{abstract}
In the nearly unstable heavy-tailed regime of Jaisson and Rosenbaum, where the intensity of a linear Hawkes process converges to a rough square-root Volterra process, we ask what a record of length $T$ reveals about the relative distance to criticality $1-\rho$ and the exogenous rate $\mu$. For the first, the information is governed by $\mathcal I_T=T\mu(1-\rho)$, not by the number of events; $\mathcal I_T$ stays bounded although the event count diverges like $T^{2\gamma}$, so no estimator of the relative distance to criticality is uniformly locally consistent. For the second, the Fisher information diverges because the rough limit has an atom at zero: the asymptotic information geometry is singular in the exogenous direction and, unlike the light-tailed case, has no Feller-type threshold. We prove that the joint maximum likelihood estimator of $(\mu,\rho)$, whose likelihood is jointly concave, recovers $\mu$ consistently at the random rate $\{Q_T/\log\log Q_T\}^{1/2}$ given by the observed information $Q_T$, while its estimate of the relative margin is only tight: the exogenous rate is learned although the endogenous parameter it is coupled to is not. For empty-start records $Q_T$ diverges because the rough limit started at zero spends positive time at zero; on windows observed long after the start, still with the branching ratio unknown, learnability holds with probability tending to one as the window grows, through the ergodicity of the rough limit and the atom of its stationary law. We also prove finite-dimensional convergence of the stationary intensity to the stationary rough Volterra process.
\end{abstract}
\begin{keyword}
Hawkes processes \sep nearly unstable processes \sep heavy tails \sep Fisher information \sep identifiability \sep rough Volterra processes
\MSC[2020] 60G55 \sep 60F17 \sep 62M09
\end{keyword}
\end{frontmatter}
\section{Introduction}\label{sec:intro}

Whether self-exciting systems operate close to criticality has been debated for more than a decade, particularly in finance \citep{HardimanBercotBouchaud2013,FilimonovSornette2015}. For a linear Hawkes process with baseline $\mu$ and kernel of mass $\rho<1$, the branching margin $1-\rho$ is the distance to the critical boundary. In a companion paper \citep{HerreraMarinP1} we obtained the minimax rate for the branching ratio under long memory at fixed $\rho<1$. Here we ask how much a record can say about the margin as the system approaches criticality.

Our results are organised around one contrast, established within a single experiment, which we call asymmetric learnability: the relative distance to criticality cannot be learned, whereas the exogenous rate, which enters the same likelihood, can. Our first result identifies the relevant amount of information. For any common normalised kernel shape, two models with the same baseline and margins differing by a factor $1+\delta$ have Kullback--Leibler divergence at most of order $\delta^2\mathcal I_T$, where
\[
\mathcal I_T=T\mu_T(1-\rho_T)\ \asymp\ \E N_T\,(1-\rho_T)^2 .
\]
More events do not imply more information about the relative margin: each event carries information of order $(1-\rho_T)^2$. In the nearly unstable heavy-tailed regime of \citet{JaissonRosenbaum2016}, $\E N_T\asymp T^{2\gamma}\to\infty$ while $\mathcal I_T$ is constant, so the information saturates and no estimator is uniformly consistent for the relative margin.

Our second result makes the contrast between the two structural quantities exact. The Fisher information of $(\log\mu,\log(1-\rho))$ can be written in terms of $\mathcal I_T$, the mean of the normalised intensity $V=\lambda/\bar\lambda$, and the boundary-information coefficient
\[
B_T=(1-\rho)\,\frac1T\E\int_0^T V(t)^{-1}\,\dd t
=\frac1T\E\int_0^T\frac{\mu}{\lambda(t)}\,\dd t.
\]
The efficient information about the margin is at most $\mathcal I_T/\rho^2$, whatever the dynamics, whereas the baseline information is asymptotically proportional to $\mathcal I_TB_T/(1-\rho)$. In the scaling limit the normalised intensity converges weakly to a rough square-root Volterra process \citep{HorstXuZhang2023,ElEuchRosenbaum2019}. Rough square-root Volterra laws have positive boundary mass \citep{FriesenGerholdWiedermann2026,BourdonJeannin2026}; this implies divergence of the reciprocal occupation functional for empty-start records, and therefore divergence of exogenous information while endogenous-margin information saturates. The remaining question is the rate. Our numerical analysis shows that the prelimit law is multiscale: the bulk can already move on the macroscopic scale while a much thinner lower layer controls reciprocal moments and Fisher information.

Our third result turns this information into estimation, without assuming the branching ratio known. The likelihood is jointly concave in $(\mu,\rho)$. We prove that the joint maximum likelihood estimator recovers $\mu$ consistently in relative terms, at the random rate $\{Q_T/\log\log Q_T\}^{1/2}$, where $Q_T=\mu^2\int_0^T\lambda_t^{-1}\dd t$ is the observed information, while its estimate of the relative margin $1-\rho$ is tight but, by the first result, cannot be consistent (Theorem~\ref{thm:nonoracle}). The mechanism is that the exogenous direction only needs the margin to within its own order of magnitude, which the likelihood provides (Proposition~\ref{prop:tight}), while the information in that direction diverges. The key probabilistic input is that the rough square-root limit started at zero spends positive time at zero almost surely (Lemma~\ref{lem:layer}). The result does not rest on this initial layer: on windows observed long after the start, the exogenous rate is learned with probability tending to one as the window grows (Theorem~\ref{thm:windows}), because the stationary rough limit is ergodic and its law has an atom at zero; for a fixed window, learnability holds only with positive probability, and we explain why this is intrinsic. To our knowledge this is the first estimation result for the exogenous rate of a nearly unstable Hawkes process in the heavy-tailed regime; its normalisation is random and set by the boundary behaviour of a rough process, with no analogue in the ergodic or local-to-unity settings.

Our fourth result addresses stationarity at the process level.  We prove finite-dimensional convergence of the stationary, rescaled Hawkes intensity to the stationary rough square-root Volterra limit by combining a Poisson-cluster Laplace bound for the remote past with the heavy-tailed intensity convergence and long-time Volterra stationarity results.  This avoids any total-variation mixing assumption and yields the stationary exogenous/endogenous information separation for the likelihood conditional on the pre-sample history.

\parhead{Related work.} Inference on the branching ratio under a time-varying baseline, including tests against alternatives approaching criticality, has been developed in an in-fill asymptotic framework \citep{PotironScailletVolkovYu2025}; asymptotic theory for maximum likelihood exists for ergodic Hawkes processes \citep{ClinetYoshida2017}. Here the observation horizon, the tail of the kernel and the approach to instability are tied together by the scaling of \citet{JaissonRosenbaum2016}, which leads to a triangular-array experiment of a different nature. The closest precedents are discrete. \citet{WeiWinnicki1990} study the joint estimation of the offspring and immigration means of critical branching processes with immigration, the discrete analogues of $\rho$ and $\mu$, and show that their estimator of the immigration mean is not consistent in the critical case; along nearly critical triangular arrays whose offspring variance tends to zero, \citet[Theorem~3.1 and Remark~3.1]{IspanyPapvanZuijlen2005} show that the estimator of the immigration mean is consistent again. Estimability of the exogenous parameter near criticality therefore depends on the regime, and our contribution is not this possibility but a different mechanism for it. Nearly unstable integer-valued autoregressions are treated by \citet{IspanyPapvanZuijlen2003} and \citet{DrostvandenAkkerWerker2009}. These light-tailed models lead to classical diffusion limits, of Cox--Ingersoll--Ross or Ornstein--Uhlenbeck type, or to Poissonian limit experiments. Heavy-tailed excitation produces a different limit: a non-Markovian rough Volterra process, whose occupation of the boundary governs the exogenous information and which, unlike the light-tailed Hawkes limit (Remark~\ref{rem:feller}), has no Feller-type threshold. The saturation of endogenous information has a classical counterpart in time series: for the local-to-unity autoregression $\rho=1-c/T$, the localising parameter $c$ is not consistently estimable \citep{Phillips1987,ChanWei1987}, which is why confidence intervals for it remain nondegenerate \citep{Stock1991}. Theorem~\ref{thm:A} extends this phenomenon to self-exciting counting processes, where the number of events diverges super-linearly; saturation also holds in the light-tailed nearly unstable regime of \citet{JaissonRosenbaum2015}, where $1-\rho_T\asymp T^{-1}$ with fixed baseline also gives a bounded $\mathcal I_T$. What is specific to the rough regime is the behaviour of the exogenous direction.

\section{The information index}\label{sec:index}

Let $N$ be a linear Hawkes process with intensity $\lambda(t)=\mu+\rho\int_{(s_0,t)}g(t-s)N(\dd s)$, where $g\ge0$ is a normalised kernel shape, $\int g=1$, and $\rho<1$, started empty at time $s_0\le0$ and observed on $[0,T]$. Write $\bar\lambda=\mu/(1-\rho)$ and $\mathcal I_T=T\mu(1-\rho)$.

\begin{lemma}\label{lem:mean}
The mean intensity of the process started empty is non-decreasing in time and bounded by $\bar\lambda$; hence $\E N(s_0,T]\le\bar\lambda(T-s_0)$.
\end{lemma}

\begin{proof}
The mean intensity is $\mu(1+\int_0^{t-s_0}R)$, with $R=\sum_{n\ge1}(\rho g)^{*n}\ge0$ and $\int R=\rho/(1-\rho)$. \qedhere
\end{proof}

Let $P_0=(\mu,\rho_0)$ and $P_1=(\mu,\rho_1)$ share $\mu$, the kernel shape $g$ and $s_0=-cT$, with $1-\rho_1=(1+\delta)(1-\rho_0)$.

\begin{theorem}\label{thm:A}
For all $\delta>0$ with $\rho_1>0$ and all $c\ge0$,
\[
\KL\big(P_0^T\,\|\,P_1^T\big)\ \le\ \frac{\delta^2(1+c)}{\rho_0\rho_1}\ \mathcal I_T ,\qquad \mathcal I_T=T\mu(1-\rho_0).
\]
\end{theorem}

\begin{proof}
Both processes start empty at $-cT$ and are described by intensities with respect to the same filtration, so the divergence on $(-cT,T]$ equals $\E_0\int_{-cT}^T\varphi(\lambda_0,\lambda_1)\,\dd t$ with $\varphi(x,y)=x\log(x/y)-x+y$ \citep{Jacod1975}. Pathwise, $\lambda_1=\mu+(\rho_1/\rho_0)(\lambda_0-\mu)$, so $\lambda_1-\lambda_0=-(m/\rho_0)(\lambda_0-\mu)$ with $m=\delta(1-\rho_0)$, and $\lambda_1\ge(\rho_1/\rho_0)\lambda_0$. Since $\varphi(x,y)\le(x-y)^2/y$, $\varphi(\lambda_0,\lambda_1)\le m^2\lambda_0/(\rho_0\rho_1)$. By Lemma~\ref{lem:mean}, the divergence is at most $m^2\bar\lambda(1+c)T/(\rho_0\rho_1)$, and $m^2\bar\lambda T=\delta^2\mathcal I_T$. Restricting to $[0,T]$ does not increase the divergence. \qedhere
\end{proof}

\begin{corollary}[Local lower bound]\label{cor:local}
Let $\mathcal P_T$ be any family containing $P_0$ and the models $(\mu,\rho_1)$ with $1-\rho_1=(1+\delta)(1-\rho_0)$, $0<\delta\le1/2$. There are constants $a,b>0$, depending only on $c$ and a lower bound for $\rho_0$, such that for every estimator
\[
\sup_{\mathcal P_T}\PP\Big(\Big|\frac{1-\hat\rho}{1-\rho}-1\Big|\ge a\min\{1,\mathcal I_T^{-1/2}\}\Big)\ \ge\ b .
\]
In particular, if $\mathcal I_T$ stays bounded, no estimator of the relative margin is uniformly locally consistent over families containing these perturbations.
\end{corollary}

\begin{proof}
Take $\delta=\min\{1/2,\ \eta\,\mathcal I_T^{-1/2}\}$ with $\eta$ small enough that the bound of Theorem~\ref{thm:A} is at most $1/8$; then $\TV(P_0^T,P_1^T)\le1/4$ by Pinsker's inequality, and Le Cam's two-point lemma \citep[Section~2.4]{Tsybakov2009} applied to the relative margin, which differs by the factor $1+\delta$, gives the claim. \qedhere
\end{proof}

\begin{remark}[Two regimes]
For fixed $\rho<1$, $\mathcal I_T\asymp T$ and the bound is the parametric $T^{-1/2}$. Uniform local relative consistency over such families requires $\mathcal I_T\to\infty$; the index counts events weighted by $(1-\rho)^2$, since $\mathcal I_T=\bar\lambda T(1-\rho)^2$. In practice this yields a simple diagnostic: for a stationary record, $\mathcal I_T\approx N_T(1-\rho)^2$, so a fitted value $\hat\rho$ with $N_T(1-\hat\rho)^2$ of order one signals that the relative distance to criticality is not resolved by the data, however large $N_T$ is. For instance, $N_T=10^6$ events and $\hat\rho=0.999$ give $N_T(1-\hat\rho)^2=1$.
\end{remark}

\section{Saturation in the nearly unstable regime}\label{sec:saturation}

Let $g(t)=\gamma(1+t)^{-1-\gamma}$ with $\gamma\in(1/2,1)$, so that $1-\widehat g(z)\sim Cz^\gamma$, $C=\Gamma(1-\gamma)$, and let, following \citet{JaissonRosenbaum2016},
\begin{equation}\label{eq:scaling}
1-\rho_T=\lambda^*T^{-\gamma},\qquad \mu_T=\mu^*T^{\gamma-1}.
\end{equation}
Then $\bar\lambda_T=(\mu^*/\lambda^*)T^{2\gamma-1}$, $\E N_T\asymp T^{2\gamma}$, and
\[
\mathcal I_T=T\mu_T(1-\rho_T)=\mu^*\lambda^*\qquad\text{for every }T .
\]
By Corollary~\ref{cor:local}, the relative margin is not uniformly consistently estimable along this sequence, although the number of events diverges. The same bound holds for the process started empty at $-cT$ for any fixed $c$.

\parhead{Scaling limit.} Rescaling time by $T$, $V(u)=\lambda(uT)/\bar\lambda_T$ satisfies in the limit, in stationary form,
\begin{equation}\label{eq:limit}
V=1+K*\big[\lambda^*(1-V)\,\dd u+\sqrt{\lambda^*/\mu^*}\,\sqrt V\,\dd B\big],\qquad K(t)=\frac{t^{\gamma-1}}{C\,\Gamma(\gamma)},
\end{equation}
and after a time change its law depends on $(\mu^*,\lambda^*)$ only through $\Lambda=\mu^*C^{1/\gamma}(\lambda^*)^{-(1-\gamma)/\gamma}$.

\parhead{Stationary records.} The local lower bound above is proved for empty-start records (or starts at a fixed multiple of $T$ before observation). Extending that lower bound to the law of a stationary record alone would require controlling the information carried by the unobserved past and remains open. Section~\ref{sec:fisher} proves a different stationary result: finite-dimensional convergence of the stationary intensity and separation of the conditional Fisher information when the pre-sample history is available.

\section{Exogenous versus endogenous information}\label{sec:fisher}

Parametrise by $\theta=\log\mu$ and $\eta=\log a$, $a=1-\rho$, with the kernel shape known, and let $V=\lambda/\bar\lambda$,
\[
A_T=\frac1T\,\E\int_0^T\frac{\dd t}{V(t)},\qquad D_T=\frac1T\,\E\int_0^T V(t)\,\dd t .
\]
For the process started empty, $D_T\le1$ by Lemma~\ref{lem:mean}, and $A_TD_T\ge1$ by the Cauchy--Schwarz inequality.

\begin{theorem}\label{thm:fisher}
The Fisher information of $(\theta,\eta)$ on $[0,T]$ is
\[
\mathcal J_T=\mathcal I_T\begin{pmatrix}A_T&-\dfrac{1-aA_T}{\rho}\\[2mm]-\dfrac{1-aA_T}{\rho}&\dfrac{D_T-2a+a^2A_T}{\rho^2}\end{pmatrix}.
\]
Consequently, the efficient informations are
\[
\mathcal J^{\rm eff}_{\log\mu,T}=\mathcal I_T\,\frac{A_TD_T-1}{D_T-2a+a^2A_T},\qquad
\mathcal J^{\rm eff}_{\log(1-\rho),T}=\frac{\mathcal I_T}{\rho^2}\Big(D_T-\frac1{A_T}\Big)\le\frac{\mathcal I_T}{\rho^2}.
\]
\end{theorem}

\begin{proof}
The Fisher information of a point process with intensity $\lambda_\vartheta$ is $\E\int(\nabla\lambda)(\nabla\lambda)^\top/\lambda\,\dd t$. Here $\partial_\theta\lambda=\mu$ and, since $\rho\int g(t-s)N(\dd s)=\lambda-\mu$ and $\partial_\eta\rho=-a$, $\partial_\eta\lambda=-(a/\rho)(\lambda-\mu)$. Using $\mu^2/\bar\lambda=a\mu$, $\mu/\bar\lambda=a$ and $a^2\bar\lambda=a\mu$, the three entries are $\mu^2\E\int\lambda^{-1}=\mathcal I_TA_T$, $-(\mu a/\rho)\E\int(1-\mu/\lambda)=-\mathcal I_T(1-aA_T)/\rho$ and $(a^2/\rho^2)\E\int(\lambda-2\mu+\mu^2/\lambda)=\mathcal I_T(D_T-2a+a^2A_T)/\rho^2$. The efficient informations are the Schur complements; the bound uses $D_T\le1$. \qedhere
\end{proof}

Define the boundary-information coefficient
\[
B_T:=aA_T=\frac1T\E\int_0^T\frac{\mu}{\lambda(t)}\,\dd t
=\frac1T\E\int_0^T\frac{\dd t}{W(t)},
\qquad W(t):=\frac{\lambda(t)}{\mu}=\frac{V(t)}a.
\]
Since $W\ge1$, $0<B_T\le1$. Replacing $A_T$ by $B_T/a$ in Theorem~\ref{thm:fisher} yields the following exact representation.

\begin{corollary}[Boundary-information form]\label{cor:boundaryinfo}
\[
\mathcal J_T=\mathcal I_T
\begin{pmatrix}
B_T/a&-\dfrac{1-B_T}{\rho}\\[2mm]
-\dfrac{1-B_T}{\rho}&\dfrac{D_T-2a+aB_T}{\rho^2}
\end{pmatrix},
\]
and
\[
\mathcal J^{\rm eff}_{\log\mu,T}
=\frac{\mathcal I_T}{a}
\frac{B_TD_T-a}{D_T-2a+aB_T},
\qquad
\mathcal J^{\rm eff}_{\log(1-\rho),T}
=\frac{\mathcal I_T}{\rho^2}
\left(D_T-\frac{a}{B_T}\right).
\]
For a stationary record $D_T=1$.
\end{corollary}

Along~\eqref{eq:scaling}, $a_T=\lambda^*T^{-\gamma}$ and $\mathcal I_T=\mu^*\lambda^*$. Hence, whenever $D_T$ stays bounded away from zero,
\[
\mathcal J^{\rm eff}_{\log\mu,T}\asymp T^\gamma B_T.
\]
If $B_T\to B_\Lambda>0$, the exogenous information has the maximal rate $T^\gamma$; if $B_T=T^{-\beta+o(1)}$ for $0<\beta<\gamma$, it grows at rate $T^{\gamma-\beta+o(1)}$. The rough-limit atom will imply divergence of $A_T=B_T/a_T$, but does not by itself force $B_T$ to have a positive limit. Thus existence and rate of exogenous identifiability are distinct questions.

Finally, with $G_T(q)=\E e^{-q(W_T-1)}$,
\[
\boxed{\quad B_T=\int_0^\infty e^{-q}G_T(q)\,\dd q.\quad}
\]
This identity links the Fisher geometry directly to the deterministic cluster-transform calculations below.

The endogenous bound holds for every common kernel shape and every dynamics: the information about the margin can never exceed $\mathcal I_T/\rho^2$.

\begin{proposition}\label{prop:F}
Along~\eqref{eq:scaling}, with empty start, assume that $V_T(u)=V(uT)\Rightarrow V_\infty(u)$ for almost every $u\in(0,1]$ and that $\int_0^1P(V_\infty(u)=0)\,\dd u>0$. Then $A_T\to\infty$, $\mathcal J^{\rm eff}_{\log\mu,T}\to\infty$ and $\mathcal J^{\rm eff}_{\log(1-\rho),T}\le\mu^*\lambda^*/\rho_T^2$.
\end{proposition}

\begin{proof}
For $M>0$, $f_M(v)=\min\{M,1/v\}$ ($f_M(0)=M$) is bounded and continuous on $[0,\infty)$, so $\E f_M(V_T(u))\to\E f_M(V_\infty(u))\ge M\,P(V_\infty(u)=0)$. By dominated convergence, $\liminf_TA_T\ge\liminf_T\int_0^1\E f_M(V_T(u))\,\dd u\ge M\int_0^1P(V_\infty(u)=0)\,\dd u$ for every $M$. Moreover $D_T\le1$ and, by Fatou's lemma for weak convergence, $\liminf_TD_T\ge\int_0^1\E V_\infty(u)\,\dd u>0$. Since $a^2A_T\le a_T$, the denominator in Theorem~\ref{thm:fisher} is at most $1+a_T$, so $\mathcal J^{\rm eff}_{\log\mu,T}\ge\mathcal I_T(A_TD_T-1)/(1+a_T)\to\infty$. \qedhere
\end{proof}

Both assumptions hold in the present setting. The first follows from Theorem~2.6 together with the weak-uniqueness argument of \citet{HorstXuZhang2023}; their heavy-tailed nearly unstable scaling includes the Omori kernel and the empty-start sequence used here, and the resulting limit coincides with~\eqref{eq:limit} after normalisation. The second follows from the positive atom at zero of the limit at every fixed time, $\PP(V_\infty(u)=0)\ge\exp\{-Cb\,u^{1-\gamma}\}>0$, proved in Lemma~\ref{lem:layer} below from the Riccati--Volterra bounds of \citet[Section~3.1 and Theorem~3.2]{FriesenGerholdWiedermann2026}; the existence of a finite-time boundary atom is also established by \citet{FriesenGerholdWiedermann2026} and, in the fractional case, independently by \citet{BourdonJeannin2026}. \citet[Theorem~1.10(b)]{FriesenGerholdWiedermann2026} further show that the subcritical rough limit distribution has an atom at zero.

\begin{theorem}[Stationary finite-dimensional convergence]\label{thm:stat}
Along~\eqref{eq:scaling}, define the stationary rescaled intensity process
\[
V_T^{\rm st}(u)=\frac{\lambda_T^{\rm st}(uT)}{\bar\lambda_T},\qquad u\ge0.
\]
Then $V_T^{\rm st}$ converges in finite-dimensional distributions to the stationary rough square-root Volterra process $V_\Lambda^{\rm stat}$ associated with the limiting distribution $\pi_\Lambda$ of \citet{FriesenJin2024}, where $\pi_\Lambda$ is the limit, as $u\to\infty$, of the law of the solution of the limit equation started at $x_0=0$ (limiting distributions of Volterra square-root processes may depend on the initial state). In particular,
$V_T^{\rm st}(0)\Rightarrow\pi_\Lambda$ and $\pi_\Lambda(\{0\})>0$ \citep[Theorem~1.10(b)]{FriesenGerholdWiedermann2026}.
\end{theorem}

\begin{proof}
Fix $0\le u_1<\cdots<u_k$ and $s_1,\ldots,s_k\ge0$. Let $V_T^{(M)}$ denote the process with the same parameters as the stationary Hawkes process but with no immigrants before $-MT$. By the Poisson-cluster representation, the only difference between the joint Laplace transforms of $V_T^{\rm st}$ and $V_T^{(M)}$ is the contribution of immigrants before $-MT$. Since $1-e^{-\sum_i x_i}\le\sum_i x_i$ for $x_i\ge0$ and the mean excitation generated by one cluster at age $r$ is the Hawkes resolvent $R_T(r)$,
\[
\begin{aligned}
0&\le \log \E e^{-\sum_i s_i V_T^{(M)}(u_i)}-
       \log \E e^{-\sum_i s_i V_T^{\rm st}(u_i)}\\
&\le \sum_{i=1}^k s_i(1-\rho_T)
       \int_{(M+u_i)T}^{\infty}R_T(r)\,\dd r .
\end{aligned}
\]
The uniform resolvent convergence of \citet[Proposition~3.1]{HorstXuZhang2023} implies that the right-hand side converges, as $T\to\infty$, to
$\sum_i s_i\{1-F(M+u_i)\}\le(\sum_i s_i)\{1-F(M)\}$, where $F$ is the corresponding Mittag--Leffler distribution function.

For fixed $M$, time translation identifies
$(V_T^{(M)}(u_1),\ldots,V_T^{(M)}(u_k))$ with the empty-start rescaled process evaluated at $(M+u_1,\ldots,M+u_k)$. The process convergence established by \citet[Theorem~2.6 and the subsequent weak-uniqueness argument]{HorstXuZhang2023} therefore yields convergence to $(V_\infty(M+u_1),\ldots,V_\infty(M+u_k))$. Finally, the shifted limit process $(V_\infty(M+u))_{u\ge0}$ converges weakly on $C(\mathbb R_+)$, as $M\to\infty$, to the stationary process $V_\Lambda^{\rm stat}$ \citep{FriesenJin2024}; see also \citet[Section~1.2]{BenAlayaFriesenKremer2026}. This applies because, in the notation $X=x_0+K*(b+\beta X)\,\dd t+\sigma K*\sqrt X\,\dd B$ of those papers, the limit equation has fractional kernel $K(t)=t^{\gamma-1}/\Gamma(\gamma)$, which satisfies their conditions (K1)--(K2) for $\gamma\in(1/2,1)$ \citep[Example~4.1]{BenAlayaFriesenKremer2026}, initial value $x_0=0$, and mean-reverting drift $\beta=-\lambda^*/\Gamma(1-\gamma)<0$ (after the normalisation of \citet[Theorem~2.7]{HorstXuZhang2023}). Letting first $T\to\infty$ and then $M\to\infty$ in the joint Laplace transforms proves the finite-dimensional convergence. \qedhere
\end{proof}

The proof avoids total-variation mixing of path laws. It uses only that the contribution of the far past becomes negligible at the macroscopic intensity scale. For one marginal the error bound is $s\{1-F(M)\}$ uniformly in $T$. The term $(1-\rho_T)\int_{MT}^{\infty}R_T$ is also the relative deficit of the mean intensity of the process started empty at $-MT$: it equals $0.4019$ for $M=4$, $T=10^5$ against $1-F(4)=0.4033$, and $0.1998$ for $M=16$ against $1-F(16)=0.2002$. Since $1-F(M)\asymp M^{-\gamma}$, macroscopic burn-in is polynomially slow.

\begin{corollary}[Information separation for stationary records]\label{cor:stat}
For the stationary process and the likelihood conditional on the pre-sample history, $D_T=1$, $A^{\rm st}_T=\E[1/V^{\rm st}_T]\to\infty$, and
\[
\mathcal J^{\rm eff}_{\log\mu,T}=\mathcal I_T\,\frac{A^{\rm st}_T-1}{1-2a_T+a_T^2A^{\rm st}_T}\to\infty,\qquad
\mathcal J^{\rm eff}_{\log(1-\rho),T}=\frac{\mathcal I_T}{\rho_T^2}\Big(1-\frac1{A^{\rm st}_T}\Big)\le\frac{\mu^*\lambda^*}{\rho_T^2}.
\]
\end{corollary}

\begin{proof}
By stationarity, time averages equal marginal expectations. The truncation argument of Proposition~\ref{prop:F}, applied to the marginal convergence of Theorem~\ref{thm:stat}, gives $\liminf_TA^{\rm st}_T\ge M\pi_\Lambda(\{0\})$ for every $M$; the formulas follow from Theorem~\ref{thm:fisher} with $D_T=1$, and $a_T^2A^{\rm st}_T\le a_T$. \qedhere
\end{proof}

\begin{remark}[Scope]\label{rem:scope}
(i) The endogenous bound also holds for the likelihood of the record alone, since discarding the pre-sample history cannot increase information. The divergence of the exogenous information is proved for the conditional likelihood; for the record alone it remains open, although the information on trimmed windows $[T/2,T]$ in Table~\ref{tab:fisher} points the same way.
(ii) The argument of Theorem~\ref{thm:stat} fails at the scale of the boundary layer: with $\theta=q/\mu_T$ the bound on the contribution of the remote past becomes $q\int_{MT}^\infty R\asymp qT^\gamma(1-F(M))/\lambda^*$, which is not small. The law of $W_T=\lambda/\mu_T$ in stationarity therefore depends on the remote past, which is why the rate coefficient $B_T$ is harder to control for stationary than for empty-start records, and why stationary transform computations are stiff.
\end{remark}

\begin{remark}[Singular information in the exogenous direction]\label{rem:singular}
For the square-root Volterra process $X=x_0+K*(b+\beta X)\,\dd t+\sigma K*\sqrt X\,\dd B$ observed continuously, the Fisher information per unit time of the drift parameters $(b,\beta)$ is, up to the factor $\sigma^{-2}$, the matrix with entries $\int x^{-1}\pi(\dd x)$, $1$, $1$ and $\int x\,\pi(\dd x)$ \citep[Theorem~3.2]{BenAlayaFriesenKremer2026}; their estimation results require the negative-moment condition $\sup_t\E X_t^{-1-\varepsilon}<\infty$, which they note is open for the fractional kernel. Theorem~\ref{thm:fisher} has the same structure as $a_T\to0$: $A_T$ plays the role of $\int x^{-1}\pi(\dd x)$, $D_T$ that of $\int x\,\pi(\dd x)$, and the intercept $b$ corresponds to the exogenous rate. Since $\pi_\Lambda$ has an atom at zero, $\int x^{-1}\pi_\Lambda(\dd x)=\infty$ and the negative-moment condition fails: the information of the limit model is infinite in the direction of the intercept. We do not derive a Le Cam limit experiment; the statement is at the level of Fisher information. Corollary~\ref{cor:stat} describes how the Hawkes prelimit regularises this singularity, with finite but divergent exogenous information $\mathcal I_T A_T$, while the information in the mean-reversion direction, the endogenous margin, stays bounded.
\end{remark}

\begin{remark}[Classical versus rough]\label{rem:feller}
The atom is sufficient but not necessary in Proposition~\ref{prop:F}: since $v\mapsto1/v$ is lower semicontinuous and nonnegative, the portmanteau theorem and Fatou's lemma give $\liminf_TA_T\ge\int_0^1\E[1/V_\infty(u)]\,\dd u$, so the exogenous information diverges whenever the limit law has an infinite inverse moment. This yields a sharp contrast. In the light-tailed nearly unstable regime the limit is a Cox--Ingersoll--Ross process \citep{JaissonRosenbaum2015}, whose stationary law is Gamma with some shape $\kappa$; its inverse moment is infinite exactly when $\kappa\le1$, a Feller-type condition, so exogenous information diverges only on one side of a threshold (the converse direction, boundedness for $\kappa>1$, would require uniform integrability of $1/V_T$, which we do not prove). In the rough regime the boundary atom exists for every parameter and no Feller-type condition exists \citep[Section~3 and Theorem~1.10]{FriesenGerholdWiedermann2026}: for the rough square-root limits considered here, roughness removes the inverse-moment threshold at the level of the conditional Fisher information, which diverges in the exogenous direction for every value of the parameters.
\end{remark}

The divergence of $A_T$ concerns the \emph{expected} information. The observed information, which normalises the score, is random: with $\rho$ and the kernel shape known, the score for $\log\mu$ is the martingale $\int_0^T(\mu/\lambda)(\dd N-\lambda\,\dd t)$, whose jumps $\mu/\lambda(t_i)$ are bounded by one and whose bracket is
\[
Q_T=\mu_T^2\int_0^T\frac{\dd t}{\lambda_T(t)}=\mathcal I_T\int_0^1\frac{\dd u}{V_T(u)} .
\]

\begin{proposition}[Observed exogenous information]\label{prop:observed}
Along~\eqref{eq:scaling} with empty start, let $L_0=\mathrm{Leb}\{u\in[0,1]:V_\infty(u)=0\}$. Then, for every $K>0$,
\[
\liminf_{T\to\infty}\PP(Q_T\ge K)\ \ge\ \PP(L_0>0)\ \ge\ \int_0^1\PP\big(V_\infty(u)=0\big)\,\dd u\ >\ 0 .
\]
If moreover $\int_0^1\dd u/V_\infty(u)=\infty$ almost surely, then $Q_T\to\infty$ in probability.
\end{proposition}

\begin{proof}
By \citet[Theorem~2.6]{HorstXuZhang2023}, $V_T\Rightarrow V_\infty$ in the Skorokhod space with a continuous limit. By the Skorokhod representation theorem we may assume $\sup_{u\le1}|V_T(u)-V_\infty(u)|\to0$ almost surely, which does not change the law of $Q_T$. Then $\liminf_T1/V_T(u)\ge1/V_\infty(u)\in(0,\infty]$ for every $u$, and Fatou's lemma gives $\liminf_T\int_0^1\dd u/V_T\ge X:=\int_0^1\dd u/V_\infty$ almost surely. Since $\liminf_T\mathbf 1\{Y_T>c\}\ge\mathbf 1\{\liminf_TY_T>c\}$, Fatou's lemma again gives $\liminf_T\PP(Q_T>K)\ge\PP(X>K/\mathcal I_T)$, and $\mathcal I_T=\mu^*\lambda^*$ is constant. On $\{L_0>0\}$ we have $X=\infty$. Finally $0\le L_0\le1$, so $\PP(L_0>0)\ge\E L_0=\int_0^1\PP(V_\infty(u)=0)\,\dd u$ by Fubini's theorem, which is positive by the atom at zero at every positive time. The last statement follows in the same way. \qedhere
\end{proof}

Proposition~\ref{prop:observed} turns the divergence of the expected information into a statement about the random bracket that normalises the score. On the event where the bracket diverges, the jumps of the score are negligible relative to it, which is the Lindeberg condition for a self-normalised martingale limit; a mixed-normal limit for the baseline estimator, as in local-to-unity autoregression, is the natural next step. That $\int_0^1\dd u/V_\infty(u)=\infty$ almost surely, and hence $Q_T\to\infty$ in probability, is a property of the rough square-root Volterra process alone; it is established in Lemma~\ref{lem:layer} and Corollary~\ref{cor:Qinf}.

Proposition~\ref{prop:F} gives divergence without a rate; the rate is governed by the boundary layer discussed in Section~\ref{sec:baseline}. We first turn to estimation.

\section{Estimating the exogenous rate}\label{sec:estimation}

The previous section shows that the information about the exogenous rate diverges. We now show that it can be used. We first treat the branching ratio as known (Theorem~\ref{thm:estimation}) and then remove this assumption (Theorem~\ref{thm:nonoracle}); in both cases the maximum likelihood estimator of $\mu_T$ is consistent in relative terms under a random normalisation set by the time the rough limit spends at zero. Together with Corollary~\ref{cor:local}, this gives a contrast within one experiment: the endogenous margin is not consistently estimable, the exogenous rate is, even when the margin is unknown.

\parhead{The initial layer.} We first strengthen the positive-probability statement of Proposition~\ref{prop:observed}.

\begin{lemma}\label{lem:layer}
Let $X$ solve $X=K*(b+\beta X)\,\dd t+\sigma K*\sqrt X\,\dd B$ with $X_0=0$, $K(t)=c\,t^{\gamma-1}$, $c>0$, $\gamma\in(1/2,1)$, $b>0$, $\beta<0$, $\sigma>0$. Then $t\mapsto\PP(X_t=0)$ is nonincreasing, there is $C<\infty$ with $\PP(X_t=0)\ge\exp\{-Cb\,t^{1-\gamma}\}$ for all $t>0$, in particular $\lim_{t\downarrow0}\PP(X_t=0)=1$, and almost surely $\mathrm{Leb}\{s\in[0,\varepsilon]:X_s=0\}>0$ for every $\varepsilon>0$.
\end{lemma}

\begin{proof}
For $q\ge0$, the exponential-affine formula \citep[Theorem~6.1]{AbiJaberLarssonPulido2019} with zero initial value gives
\[
\E e^{-qX_t}=\exp\Big\{-b\int_0^t\Psi_q(s)\,\dd s\Big\},\qquad \Psi_q=qK+K*\big(\beta\Psi_q-\tfrac{\sigma^2}{2}\Psi_q^2\big),
\]
with $\Psi_q\ge0$; the term that multiplies the initial value vanishes because $X_0=0$. The function $\Psi_q$ is the function $V_q$ of \citet[Section~3.1, (3.1)--(3.2)]{FriesenGerholdWiedermann2026}, where the constant of the kernel is absorbed into $(b,\beta,\sigma)$. By \citet[Theorem~3.2]{FriesenGerholdWiedermann2026}, $q\mapsto\Psi_q(s)$ is nondecreasing and converges to the maximal solution $\bar\Psi=V_\infty$, which is finite on $(0,\infty)$. Moreover, in their Section~3.1, for $\beta<0$ the solutions are dominated by those with $\beta=0$ (a comparison based on Theorem~C.2 of \citealp{AbiJaberLarssonPulido2019}), and for $\beta=0$ uniqueness and scaling give $\lambda^{\gamma}V^{(0)}_\infty(\lambda t)=V^{(0)}_\infty(t)$. Hence $\bar\Psi(t)\le V^{(0)}_\infty(t)=V^{(0)}_\infty(1)\,t^{-\gamma}$, which is integrable at zero because $\gamma<1$. Monotone convergence gives
\[
\PP(X_t=0)=\lim_{q\to\infty}\E e^{-qX_t}=\exp\Big\{-b\int_0^t\bar\Psi\Big\}\ge\exp\Big\{-\frac{b\,V^{(0)}_\infty(1)}{1-\gamma}\,t^{1-\gamma}\Big\},
\]
which is nonincreasing in $t$ and tends to one as $t\downarrow0$; this part does not use the existence of an atom at positive times. For the last claim, let $Z_\varepsilon=\varepsilon^{-1}\mathrm{Leb}\{s\le\varepsilon:X_s=0\}\in[0,1]$. By Fubini's theorem and monotonicity, $\PP(Z_\varepsilon>0)\ge\E Z_\varepsilon=\varepsilon^{-1}\int_0^\varepsilon\PP(X_s=0)\,\dd s\ge\PP(X_\varepsilon=0)$. The events $\{Z_\varepsilon>0\}$ increase with $\varepsilon$, so $\PP(Z_{\varepsilon'}>0)\ge\sup_{\varepsilon\le\varepsilon'}\PP(X_\varepsilon=0)=1$ for every $\varepsilon'>0$. \qedhere
\end{proof}

\begin{corollary}\label{cor:Qinf}
Along~\eqref{eq:scaling} with empty start, $Q_T\to\infty$ in probability.
\end{corollary}

\begin{proof}
The limit $V_\infty$ of the empty-start process is, up to a constant factor, a process as in Lemma~\ref{lem:layer} with $b>0$ proportional to $\mu^*$ \citep[Theorem~2.7]{HorstXuZhang2023}. By the lemma, $L_0>0$ almost surely, so $\int_0^1\dd u/V_\infty=\infty$ almost surely, and Proposition~\ref{prop:observed} gives the claim. \qedhere
\end{proof}

The divergence in probability comes from the initial layer of the empty-start record, where the intensity starts at its floor; on a fixed window $[\eta T,T]$ away from the start, Proposition~\ref{prop:observed} gives divergence only with positive probability, and Theorems~\ref{thm:windows} and~\ref{thm:windows-nonoracle} below show that on long windows it occurs with probability tending to one, through a different mechanism.

\parhead{Consistency at a random rate.} Assume that the kernel shape $g$ and the branching ratio $\rho_T$ are known. Write $E_i=\rho_T\sum_{j<i}g(t_i-t_j)$ for the excitation at the $i$th event, which does not depend on $\mu$. The log-likelihood is
\[
\ell_T(m)=\sum_{i:\,t_i\le T}\log(m+E_i)-mT,\qquad m>0,
\]
which is strictly concave. Its derivative $S_T(m)=\sum_i(m+E_i)^{-1}-T$ is strictly decreasing, tends to $+\infty$ as $m\downarrow0$ when $N(0,T]\ge1$ (because $E_1=0$) and to $-T$ as $m\to\infty$. Hence the maximum likelihood estimator $\hat\mu_T$, the unique root of $S_T$, exists whenever $N(0,T]\ge1$, an event of probability $1-e^{-\mu_TT}\to1$.

\begin{theorem}\label{thm:estimation}
Along~\eqref{eq:scaling} with empty start, and with $g$ and $\rho_T$ known, let $\mathsf L(q)=\log\log(q\vee e^e)$. Then
\[
\Big(\frac{Q_T}{\mathsf L(Q_T)}\Big)^{1/2}\Big(\frac{\hat\mu_T}{\mu_T}-1\Big)=O_P(1),\qquad Q_T\to\infty\ \text{in probability}.
\]
In particular $\hat\mu_T/\mu_T\to1$ in probability, while no estimator of the relative margin $1-\rho_T$ is uniformly locally consistent (Corollary~\ref{cor:local}).
\end{theorem}

\begin{proof}
Let $x_i=\mu_T/\lambda_T(t_i)\in(0,1]$, where $\lambda_T(t_i)=\mu_T+E_i$ is the true intensity. The process $M_t=\sum_{t_i\le t}x_i-\mu_Tt=\int_0^t(\mu_T/\lambda_T)(\dd N-\lambda_T\dd s)$ is a martingale with jumps in $(0,1]$ and predictable quadratic variation $Q_t=\mu_T^2\int_0^t\dd s/\lambda_T$; similarly $M'_t=\sum_{t_i\le t}x_i^2-Q_t$ is a martingale with jumps in $(0,1]$ and predictable quadratic variation $\int_0^t\mu_T^4\lambda_T^{-3}\dd s\le Q_t$. For $|\delta|<1$ the identity $x/(1+\delta x)=x-\delta x^2+\delta^2x^3/(1+\delta x)$ gives
\[
\mu_TS_T\big(\mu_T(1+\delta)\big)=M_T-\delta\,(Q_T+M'_T)+\delta^2\sum_i\frac{x_i^3}{1+\delta x_i}.
\]

\emph{Step 1 (self-normalised bounds).} By the exponential inequality for counting-process martingales with bounded integrands \citep[Lemma~2.1]{vandeGeer1995}, $\PP(|M_T|\ge a,\ Q_T\le v)\le2\exp\{-a^2/(2(a+v))\}$, and the same bound holds for $M'$. Let $A_j=\{2^j\le Q_T<2^{j+1}\}$ and $a_j=4(2^{j+1}\log j)^{1/2}$. For $j$ large, $a_j\le2^{j+1}$, so $\PP(|M_T|\ge a_j,A_j)\le2\exp\{-a_j^2/2^{j+3}\}=2j^{-4}$. On $A_j$, $a_j\le C_0(Q_T\mathsf L(Q_T))^{1/2}$ for an absolute constant $C_0$. Summing over $j\ge j_0$,
\[
\PP\big(|M_T|\vee|M'_T|>C_0(Q_T\mathsf L(Q_T))^{1/2},\ Q_T\ge2^{j_0}\big)\le4\sum_{j\ge j_0}j^{-4}=:\varepsilon_{j_0}.
\]

\emph{Step 2 (trapping the root).} Let $C=C_0$ and fix $q_0$ such that $C(q\mathsf L(q))^{1/2}\le q/4$ and $\kappa(q):=4C(\mathsf L(q)/q)^{1/2}\le1/4$ for $q\ge q_0$. On the event $G$ where $Q_T\ge\max(q_0,2^{j_0})$ and $|M_T|\vee|M'_T|\le C(Q_T\mathsf L(Q_T))^{1/2}$, we have $\sum_ix_i^2=Q_T+M'_T\ge3Q_T/4$. Put $\kappa=\kappa(Q_T)$. For $\delta=\kappa$ the last term is at most $\kappa^2\sum_ix_i^2$, so
\[
\mu_TS_T\big(\mu_T(1+\kappa)\big)\le M_T-\kappa(1-\kappa)\sum_ix_i^2\le C(Q_T\mathsf L)^{1/2}-\tfrac{9}{4}C(Q_T\mathsf L)^{1/2}<0 .
\]
For $\delta=-\kappa$ the last term is nonnegative, so
\[
\mu_TS_T\big(\mu_T(1-\kappa)\big)\ge M_T+\kappa\sum_ix_i^2\ge-C(Q_T\mathsf L)^{1/2}+3C(Q_T\mathsf L)^{1/2}>0 .
\]
Since $S_T$ is decreasing, $|\hat\mu_T/\mu_T-1|<\kappa(Q_T)$ on $G$. By Corollary~\ref{cor:Qinf} and Step~1, $\liminf_T\PP(G)\ge1-\varepsilon_{j_0}$, and $\varepsilon_{j_0}\to0$ as $j_0\to\infty$. \qedhere
\end{proof}

\begin{remark}[Observable normalisation]\label{rem:estimation}
The normalisation $Q_T$ involves the unknown $\mu_T$, but its plug-in version $\hat Q_T=\sum_i\hat\mu_T^2/(\hat\mu_T+E_i)^2$ satisfies $\hat Q_T/Q_T\to1$ in probability, so the normalisation is observable. Indeed, $\sum_ix_i^2=Q_T+M'_T$ and $M'_T/Q_T=O_P((\mathsf L(Q_T)/Q_T)^{1/2})=o_P(1)$ by Step~1 of the proof, so $\sum_ix_i^2=Q_T\{1+o_P(1)\}$. Writing $\hat\mu_T=\mu_T(1+\hat\delta)$ with $|\hat\delta|\le1/2$, each summand satisfies $\hat x_i=\hat\mu_T/(\hat\mu_T+E_i)=x_i(1+\hat\delta)/(1+\hat\delta x_i)$, and since $x_i\in(0,1]$ the factor satisfies $1-|\hat\delta|\le(1+\hat\delta)/(1+\hat\delta x_i)\le1+|\hat\delta|$ whatever the sign of $\hat\delta$; hence $(1-|\hat\delta|)^2\le\hat Q_T/\sum_ix_i^2\le(1+|\hat\delta|)^2$, uniformly in $i$, and $\hat\delta\to0$ in probability by the theorem.
\end{remark}

\parhead{Unknown branching ratio.} Theorem~\ref{thm:estimation} assumes $\rho_T$ known, which is precisely the parameter that Corollary~\ref{cor:local} shows cannot be recovered in relative scale. We now remove this assumption. Fix $r_0\in(0,1)$ and consider the joint maximum likelihood estimator $(\hat\mu_T,\hat\rho_T)$ over $(0,\infty)\times[r_0,1]$ of
\[
\ell_T(m,r)=\sum_{i:\,t_i\le T}\log(m+rG_i)-mT-r\Gamma_T,
\]
where $G_i=\sum_{j<i}g(t_i-t_j)$, $\Gamma_T=\int_0^TG(t)\,\dd t$ and $G(t)=\sum_{t_j<t}g(t-t_j)$, so that $\lambda_T=\mu_T+\rho_TG$. Since the intensity is affine in $(m,r)$, $\ell_T$ is jointly concave; it tends to $-\infty$ as $m\downarrow0$ (because $G_1=0$) and as $m\to\infty$, so a maximiser exists, and it is unique as soon as two events have distinct values of $G_i$, an event of probability tending to one. For each $r$, write $\hat m(r)$ for the unique root of $\partial_m\ell_T(\cdot,r)$; then $\hat\mu_T=\hat m(\hat\rho_T)$. Throughout, $a_T=1-\rho_T$, $x_i=\mu_T/\lambda_T(t_i)$ and $M_T=\sum_ix_i-\mu_TT$ are as in the proof of Theorem~\ref{thm:estimation}.

\begin{lemma}[Perturbation of the exogenous score]\label{lem:perturb}
For $m\ge\mu_T/2$ and $r\in[r_0,1]$,
\[
\mu_T\big|\partial_m\ell_T(m,r)-\partial_m\ell_T(m,\rho_T)\big|\le\frac{2|r-\rho_T|}{r_0}\,\big(\mu_TT+|M_T|\big).
\]
In particular, if $|r-\rho_T|\le Ca_T$, the right-hand side is at most $(2C/r_0)(\mathcal I_T+a_T|M_T|)$, and $a_T|M_T|\to0$ in probability.
\end{lemma}

\begin{proof}
The difference equals $\mu_T\sum_i(\rho_T-r)G_i/\{(m+rG_i)(m+\rho_TG_i)\}$. Since $G_i/(m+rG_i)\le1/r_0$ and $\mu_T/(m+\rho_TG_i)\le2\mu_T/(\mu_T+\rho_TG_i)=2x_i$, it is bounded by $(2|r-\rho_T|/r_0)\sum_ix_i$, and $\sum_ix_i=\mu_TT+M_T$. For the last claim, $\E M_T^2=\E Q_T\le\mu_TT$ because $\lambda_T\ge\mu_T$, and $a_T(\mu_TT)^{1/2}=\lambda^*(\mu^*)^{1/2}T^{-\gamma/2}\to0$. \qedhere
\end{proof}

The perturbation is of order one, whereas the exogenous score at $\mu_T(1\pm\kappa)$ is of order $(Q_T\mathsf L(Q_T))^{1/2}\to\infty$ (proof of Theorem~\ref{thm:estimation}). It therefore suffices to know $\rho_T$ to within a multiple of its own margin $a_T$, not in relative terms. The next result shows that the joint estimator achieves this.

\begin{lemma}[Nondegenerate mass]\label{lem:mass}
Let $X$ be as in Lemma~\ref{lem:layer}, with $X_0=0$. Then $\int_0^tX_s\,\dd s>0$ almost surely for every $t>0$.
\end{lemma}

\begin{proof}
Since $X$ is continuous and nonnegative, $\int_0^tX=0$ if and only if $X\equiv0$ on $[0,t]$; call this event $A$. For fixed $s\in(0,t]$, the process $r\mapsto\int_0^rK(s-v)\sigma\sqrt{X_v}\,\dd B_v$, $r\le s$, is a continuous local martingale with quadratic variation $\int_0^rK(s-v)^2\sigma^2X_v\,\dd v$, which vanishes on $A$; hence its value at $r=s$ vanishes almost surely on $A$. On $A$ the equation therefore reduces to $0=X_s=b\int_0^sK(s-v)\,\dd v>0$, so $\PP(A)=0$. \qedhere
\end{proof}

\begin{proposition}[Tightness of the relative margin]\label{prop:tight}
Along~\eqref{eq:scaling} with empty start, $(1-\hat\rho_T)/(1-\rho_T)=O_P(1)$.
\end{proposition}

\begin{proof}
Since $\hat\rho_T\le1$, it suffices to bound $1-\hat\rho_T$ from above. For $\delta>0$ let $r_\delta=\rho_T-\delta a_T$, which lies in $[r_0,1]$ for $T$ large, and let $\ell_p(r)=\sup_m\ell_T(m,r)$, which is concave as a partial maximum of a jointly concave function. If $\ell_p(r_\delta)<\ell_p(\rho_T)$, then $\hat\rho_T>r_\delta$: otherwise $\hat\rho_T\le r_\delta<\rho_T$, so $r_\delta$ lies between $\hat\rho_T$ and $\rho_T$, and concavity gives $\ell_p(r_\delta)\ge\min\{\ell_p(\hat\rho_T),\ell_p(\rho_T)\}=\ell_p(\rho_T)$ because $\hat\rho_T$ maximises $\ell_p$, a contradiction. Hence $1-\hat\rho_T<(1+\delta)a_T$. We show that $\PP\{\ell_p(r_\delta)\ge\ell_p(\rho_T)\}$ is small for $\delta$ large.

\emph{Localisation in $m$.} For fixed $\delta$, Lemma~\ref{lem:perturb} with $C=\delta$ and the trapping argument of Theorem~\ref{thm:estimation} give, on an event whose probability tends to $1-\varepsilon_{j_0}$, that $\hat m(r_\delta)$ lies in $\mu_T(1\pm\kappa)$ with $\kappa=\kappa(Q_T)\to0$ in probability. On that event, since $\ell_p(r_\delta)=\ell_T(\hat m(r_\delta),r_\delta)$ and $\ell_p(\rho_T)\ge\ell_T(\hat m(r_\delta),\rho_T)$,
\[
\ell_p(r_\delta)-\ell_p(\rho_T)\le\sup_{|m/\mu_T-1|\le\kappa}\big\{\ell_T(m,r_\delta)-\ell_T(m,\rho_T)\big\}.
\]

\emph{Expansion in $r$.} Let $u_i(m)=\rho_TG_i/(m+\rho_TG_i)\in[0,1)$ and $\epsilon=\delta a_T/\rho_T$. Then $(m+r_\delta G_i)/(m+\rho_TG_i)=1-\epsilon u_i(m)$ and, by $\log(1-z)\le-z-z^2/2$ on $[0,1)$,
\[
\ell_T(m,r_\delta)-\ell_T(m,\rho_T)\le-\epsilon\Big(\sum_iu_i(m)-\rho_T\Gamma_T\Big)-\frac{\epsilon^2}{2}\sum_iu_i(m)^2 .
\]
At $m=\mu_T$, $u_i=1-x_i=\rho_TG(t_i)/\lambda_T(t_i)$ and the compensator of $\sum_iu_i$ is $\int_0^T\rho_TG\,\dd t=\rho_T\Gamma_T$; hence $Y_T=\sum_iu_i(\mu_T)-\rho_T\Gamma_T$ is a martingale with predictable quadratic variation at most $\Lambda_T=\int_0^T\lambda_T\,\dd t$. Since $|u_i(m)-u_i(\mu_T)|\le2\kappa x_i$ for $|m/\mu_T-1|\le\kappa\le1/2$, the right-hand side is at most
\[
\delta\,|X_T|+\delta\,\frac{2\kappa a_T}{\rho_T}\sum_ix_i-\frac{\delta^2a_T^2}{2\rho_T^2}\Big(\sum_iu_i(\mu_T)^2-4\kappa\sum_ix_i\Big),\qquad X_T=\frac{a_T}{\rho_T}Y_T .
\]

\emph{Orders of magnitude.} First, $\E X_T^2\le a_T^2\E\Lambda_T/\rho_T^2\le a_T^2\bar\lambda_TT/\rho_T^2=\mathcal I_T/\rho_T^2$ by Lemma~\ref{lem:mean}, so $X_T$ is tight. Second, $a_T\sum_ix_i=a_T(\mu_TT+M_T)=\mathcal I_T+o_P(1)$, so the middle term is $\delta\kappa\,O_P(1)=o_P(1)$. Third, $\sum_iu_i(\mu_T)^2=\sum_i(1-x_i)^2\ge N_T-2\sum_ix_i$, and $a_T^2\sum_ix_i=a_T\mathcal I_T+o_P(1)\to0$, while
\[
a_T^2N_T=\mathcal I_T\,\frac{N_T}{\bar\lambda_TT}=\mathcal I_T\Big(\int_0^1V_T(u)\,\dd u+\frac{N_T-\Lambda_T}{\bar\lambda_TT}\Big).
\]
The last fraction tends to zero in probability because $\E(N_T-\Lambda_T)^2=\E\Lambda_T\le\bar\lambda_TT$, and $\int_0^1V_T\Rightarrow\int_0^1V_\infty$ by \citet[Theorem~2.6]{HorstXuZhang2023}. Moreover $\int_0^1V_\infty>0$ almost surely by Lemma~\ref{lem:mass}. Hence for every $\varepsilon>0$ there are $\eta>0$ and $K_\varepsilon$ with $\PP(a_T^2N_T<\eta)\le\varepsilon$ and $\PP(|X_T|>K_\varepsilon)\le\varepsilon$ for $T$ large.

\emph{Conclusion.} Outside these events, and with probability tending to one for the $o_P(1)$ terms, $\ell_p(r_\delta)-\ell_p(\rho_T)\le\delta K_\varepsilon+1-\delta^2\eta/(4\rho_T^2)<0$ as soon as $\delta>8(K_\varepsilon+1)/\eta$. Thus $\limsup_T\PP\{1-\hat\rho_T\ge(1+\delta)a_T\}\le2\varepsilon+\varepsilon_{j_0}$, and $\varepsilon$, $j_0$ are arbitrary. \qedhere
\end{proof}

\begin{theorem}[Learning the exogenous rate with unknown branching ratio]\label{thm:nonoracle}
Along~\eqref{eq:scaling} with empty start and known kernel shape $g$, the joint estimator over $(0,\infty)\times[r_0,1]$ satisfies
\[
\Big(\frac{Q_T}{\mathsf L(Q_T)}\Big)^{1/2}\Big(\frac{\hat\mu_T}{\mu_T}-1\Big)=O_P(1),\qquad\frac{1-\hat\rho_T}{1-\rho_T}=O_P(1),
\]
whereas no estimator of $1-\rho_T$ is uniformly locally consistent in relative scale (Corollary~\ref{cor:local}).
\end{theorem}

\begin{proof}
Fix $C\ge1$ and let $H_C=\{1-\hat\rho_T\le Ca_T\}$. On $H_C$, $|\hat\rho_T-\rho_T|\le Ca_T$, because $\hat\rho_T\le1$ gives $\hat\rho_T-\rho_T\le a_T$ and $H_C$ gives $\rho_T-\hat\rho_T\le(C-1)a_T$. Let $G$ and $\kappa$ be as in the proof of Theorem~\ref{thm:estimation}, where $\mu_TS_T(\mu_T(1+\kappa))\le-\frac54C_0(Q_T\mathsf L)^{1/2}$ and $\mu_TS_T(\mu_T(1-\kappa))\ge2C_0(Q_T\mathsf L)^{1/2}$, with $S_T=\partial_m\ell_T(\cdot,\rho_T)$. By Lemma~\ref{lem:perturb}, on $H_C\cap G\cap\{a_T|M_T|\le1\}$,
\[
\mu_T\partial_m\ell_T\big(\mu_T(1\pm\kappa),\hat\rho_T\big)=\mu_TS_T\big(\mu_T(1\pm\kappa)\big)+O(C(\mathcal I_T+1)/r_0),
\]
and the error term is bounded while $C_0(Q_T\mathsf L)^{1/2}\to\infty$ in probability (Corollary~\ref{cor:Qinf}). Hence, with probability tending to at least $\PP(H_C)-\varepsilon_{j_0}$, the decreasing function $\partial_m\ell_T(\cdot,\hat\rho_T)$ is negative at $\mu_T(1+\kappa)$ and positive at $\mu_T(1-\kappa)$, so that $\hat\mu_T=\hat m(\hat\rho_T)\in\mu_T(1\pm\kappa)$. Proposition~\ref{prop:tight} makes $\liminf_T\PP(H_C)$ arbitrarily close to one for $C$ large. \qedhere
\end{proof}

Theorem~\ref{thm:nonoracle} is the positive half of the contrast announced in the introduction: the exogenous rate is learned although the endogenous margin, which enters the same likelihood, is not. The mechanism is that the exogenous direction only needs the margin to within its own order of magnitude, which the likelihood provides, while the information in that direction diverges.

\parhead{Windows away from the start.} The divergence in Corollary~\ref{cor:Qinf} comes from the initial layer. We now show that the exogenous rate is also learned from records observed long after the start, through a mechanism intrinsic to the regime: the atom at zero of the stationary law. Fix $M,H>0$. The process starts empty at time $0$, the interval $[0,MT]$ is a burn-in that enters the intensity but not the likelihood, and the likelihood is computed on the window $W=[MT,(M+H)T]$ conditionally on the past. Write
\[
Q_T^W=\mu_T^2\int_W\frac{\dd t}{\lambda_T(t)}=\mathcal I_T\int_M^{M+H}\frac{\dd u}{V_T(u)},
\]
and let $\hat\mu_T^W$ maximise $\ell_T^W(m)=\sum_{t_i\in W}\log(m+E_i)-mHT$ over $m\ge0$ (with $\rho_T$ known); the maximiser may be $0$ when $\partial_m\ell_T^W(0+)\le0$.

\begin{theorem}[Long windows]\label{thm:windows}
Along~\eqref{eq:scaling}, for every $K>0$ and $\varepsilon>0$,
\[
\lim_{H\to\infty}\liminf_{M\to\infty}\liminf_{T\to\infty}\PP\big(Q_T^W\ge K\big)=1
\]
and
\[
\lim_{H\to\infty}\limsup_{M\to\infty}\limsup_{T\to\infty}\PP\big(|\hat\mu_T^W/\mu_T-1|>\varepsilon_\mu\big)=0\qquad\text{for every }\varepsilon_\mu>0 .
\]
\end{theorem}

The order of the limits is part of the statement. First $T\to\infty$ with $M$ and $H$ fixed, so that the record is described by the rough limit on the macroscopic scale $T$; then $M\to\infty$, so that the observation starts many macroscopic time units after the process began and the initial layer is forgotten; finally $H\to\infty$, so that the window accumulates many macroscopic time units of boundary visits. The theorem does not assert consistency for a window of fixed macroscopic length, and Remark~\ref{rem:windows} explains why no such assertion should be expected.

\begin{proof}
\emph{Step 1 (fixed $M,H$).} By \citet[Theorem~2.6]{HorstXuZhang2023}, the empty-start process converges in the Skorokhod space on $[0,M+H]$ to a continuous limit $V_\infty$. As in the proof of Proposition~\ref{prop:observed}, the Skorokhod representation and Fatou's lemma give $\liminf_T\PP(Q_T^W\ge K)\ge\PP\big(\int_M^{M+H}\dd u/V_\infty(u)>K/\mathcal I\big)$, where $\mathcal I=\mu^*\lambda^*$.

\emph{Step 2 ($M\to\infty$).} Let $K'=K/\mathcal I+1$ and $\varphi(v)=\int_0^H h(v(s))\,\dd s$ for nonnegative $v\in C([0,H])$, where $h(x)=\min\{K',1/x\}$ for $x>0$ and $h(0)=K'$. The function $h$ is bounded by $K'$ and continuous on $[0,\infty)$, since $h\equiv K'$ on $[0,1/K']$; by dominated convergence $\varphi$ is a bounded continuous functional for the uniform topology. Since $\int_M^{M+H}\dd u/V_\infty\ge\varphi(V_\infty(M+\cdot))$, and the shifted process $V_\infty(M+\cdot)$ converges weakly on $C(\mathbb R_+)$ to the stationary process $V_\Lambda^{\rm stat}$ \citep[Section~1.2]{FriesenJin2024,BenAlayaFriesenKremer2026}, the portmanteau theorem for the open set $\{\varphi>K/\mathcal I\}$ gives
\[
\begin{aligned}
\liminf_{M\to\infty}\liminf_{T\to\infty}\PP\big(Q_T^W\ge K\big)&\ge\PP\big(\varphi(V_\Lambda^{\rm stat})>K/\mathcal I\big)\\
&\ge\PP\big(\mathrm{Leb}\{s\le H:V_\Lambda^{\rm stat}(s)=0\}\ge1\big),
\end{aligned}
\]
because $\varphi(v)\ge K'\,\mathrm{Leb}\{s\le H:v(s)=0\}$.

\emph{Step 3 ($H\to\infty$).} In the notation $X=x_0+K*(b+\beta X)\,\dd t+\sigma K*\sqrt X\,\dd B$ of \citet{BenAlayaFriesenKremer2026}, the limit $V_\infty$ has, after the normalisation of \citet[Theorem~2.7]{HorstXuZhang2023}, fractional kernel, $x_0=0$, $b>0$ proportional to $\mu^*$, $\beta=-\lambda^*/\Gamma(1-\gamma)<0$ and $\sigma>0$; the fractional kernel satisfies their conditions (K1)--(K2) \citep[Example~4.1]{BenAlayaFriesenKremer2026}. Their stationary law is \emph{defined} as the weak limit of the shifted process \citep[Section~1.2]{BenAlayaFriesenKremer2026}, so it is the law of $V_\Lambda^{\rm stat}$ in Step~2 by construction, with one-dimensional marginal the limiting distribution $\pi_\Lambda$ from $x_0=0$; no uniqueness of invariant measures is needed. This stationary process is ergodic, in the sense that its shift-invariant $\sigma$-field is trivial \citep[Corollary~2.6]{BenAlayaFriesenKremer2026}. By Birkhoff's ergodic theorem, which applies to the bounded measurable function $\mathbf 1\{x=0\}$ without continuity requirements, $H^{-1}\mathrm{Leb}\{s\le H:V_\Lambda^{\rm stat}(s)=0\}\to\pi_\Lambda(\{0\})$ almost surely, and $\pi_\Lambda(\{0\})>0$ by \citet[Theorem~1.10(b)]{FriesenGerholdWiedermann2026}. Hence $\PP(\mathrm{Leb}\{s\le H:V_\Lambda^{\rm stat}(s)=0\}\ge1)\to1$, which proves the first statement.

\emph{Step 4 (estimation).} On the window, $M^W_t=\sum_{t_i\in[MT,t]}x_i-\mu_T(t-MT)$ and $M'^W_t=\sum_{t_i\in[MT,t]}x_i^2-\mu_T^2\int_{MT}^t\lambda_T^{-1}$ are martingales with jumps in $(0,1]$ and predictable quadratic variations at most $Q_T^W$, and the score satisfies the identity of the proof of Theorem~\ref{thm:estimation} with $(M_T,M'_T,Q_T)$ replaced by $(M^W,M'^W,Q_T^W)$. Steps~1 and~2 of that proof therefore show that, outside an event of probability at most $\varepsilon_{j_0}$, $|\hat\mu_T^W/\mu_T-1|<\kappa(Q_T^W)$ whenever $Q_T^W\ge\max(q_0,2^{j_0})$; in that case $\partial_m\ell_T^W(\mu_T(1-\kappa))>0$, so $\hat\mu_T^W$ is interior. Choosing $K$ with $\kappa(q)\le\varepsilon$ for $q\ge K$ and applying the first statement proves the second, since $j_0$ is arbitrary. \qedhere
\end{proof}

\begin{remark}[Why long windows]\label{rem:windows}
For a fixed window the argument of Proposition~\ref{prop:observed} gives divergence only with probability at least the average atom mass over the window,
\[
H^{-1}\int_M^{M+H}\PP(V_\infty(u)=0)\,\dd u>0,
\]
and we do not expect probability one: the stationary law has unbounded support, and from a high level the square-root noise, of order the square root of the level, is small relative to the level, so the limit can stay away from zero over a fixed window with positive probability. On such paths the window carries bounded exogenous information and the estimator can even sit at the boundary $\hat\mu_T^W=0$. Learnability away from the start is therefore a long-window property, governed by the stationary atom $\pi_\Lambda(\{0\})$ through ergodicity, not by the initial layer. Theorem~\ref{thm:windows-nonoracle} below extends the non-oracle result to windows.
\end{remark}

The window estimator in Theorem~\ref{thm:windows} assumes $\rho_T$ known. The next result removes this assumption, so that the conclusion of Theorem~\ref{thm:nonoracle} also holds away from the start. Let $(\hat\mu_T^W,\hat\rho_T^W)$ maximise
\[
\ell_T^W(m,r)=\sum_{t_i\in W}\log(m+rG_i)-mHT-r\Gamma_T^W,\qquad\Gamma_T^W=\int_WG(t)\,\dd t,
\]
over $[0,\infty)\times[r_0,1]$. The function $\ell^W_T$ is jointly concave, upper semicontinuous and tends to $-\infty$ as $m\to\infty$, so a maximiser exists; $\hat\mu_T^W$ is the maximiser of the strictly concave function $\ell_T^W(\cdot,\hat\rho_T^W)$ over $m\ge0$.

\begin{theorem}[Long windows, unknown branching ratio]\label{thm:windows-nonoracle}
Along~\eqref{eq:scaling}, for every $\varepsilon_{\rm p}>0$ there is $C<\infty$ such that
\[
\limsup_{H\to\infty}\limsup_{M\to\infty}\limsup_{T\to\infty}\PP\Big(\frac{1-\hat\rho_T^W}{1-\rho_T}>C\Big)\le\varepsilon_{\rm p},
\]
and
\[
\lim_{H\to\infty}\limsup_{M\to\infty}\limsup_{T\to\infty}\PP\big(|\hat\mu_T^W/\mu_T-1|>\varepsilon\big)=0 .
\]
\end{theorem}

\begin{proof}
Write $\limsup^*$ for $\limsup_{H\to\infty}\limsup_{M\to\infty}\limsup_{T\to\infty}$ and $\liminf^*$ analogously; all constants below may depend on $H$ but are fixed before $M,T\to\infty$. We use the window analogues of the objects of Section~\ref{sec:estimation}: $x_i$, the martingales $M^W,M'^W$ with brackets at most $Q_T^W$ (Step~4 of the proof of Theorem~\ref{thm:windows}), $N^W=N(W)$ and $\Lambda^W=\int_W\lambda_T$.

\emph{Step 1 (information).} By Steps~1--2 of the proof of Theorem~\ref{thm:windows}, for every $K$ (possibly depending on $H$), $\liminf_{M}\liminf_T\PP(Q_T^W\ge K)\ge p_H:=\PP(\mathrm{Leb}\{s\le H:V_\Lambda^{\rm stat}(s)=0\}\ge1)$, and $p_H\to1$ by Step~3.

\emph{Step 2 (perturbation and trapping).} The proof of Lemma~\ref{lem:perturb} applies verbatim on $W$ and gives, for $m\ge\mu_T/2$ and $|r-\rho_T|\le Ca_T$,
\[
\mu_T\big|\partial_m\ell_T^W(m,r)-\partial_m\ell_T^W(m,\rho_T)\big|\le\frac{2C}{r_0}\big(H\mathcal I_T+a_T|M^W|\big),
\]
and $a_T|M^W|\to0$ in probability because $\E(M^W)^2\le\mu_THT$. On the event $G^W$ where $Q_T^W\ge\max(q_0,2^{j_0})$ and $|M^W|\vee|M'^W|\le C_0(Q_T^W\mathsf L(Q_T^W))^{1/2}$, Step~2 of the proof of Theorem~\ref{thm:estimation} gives $\mu_T\partial_m\ell_T^W(\mu_T(1+\kappa),\rho_T)\le-\frac54C_0(Q_T^W\mathsf L)^{1/2}$ and $\mu_T\partial_m\ell_T^W(\mu_T(1-\kappa),\rho_T)\ge2C_0(Q_T^W\mathsf L)^{1/2}$, with $\kappa=\kappa(Q_T^W)$. Hence there is $K_1(C,H)$ such that on $G^W\cap\{Q_T^W\ge K_1(C,H)\}\cap\{a_T|M^W|\le1\}$, for every $r$ with $|r-\rho_T|\le Ca_T$, the decreasing function $\partial_m\ell_T^W(\cdot,r)$ is negative at $\mu_T(1+\kappa)$ and positive at $\mu_T(1-\kappa)$; its maximiser over $m\ge0$ then lies in $\mu_T(1\pm\kappa)$. By Step~1 and Step~1 of the proof of Theorem~\ref{thm:estimation}, the probability of this event has $\liminf^*$ at least $1-\varepsilon_{j_0}$, and $j_0$ is arbitrary.

\emph{Step 3 (tightness).} Let $\delta\ge1$ and $r_\delta=\rho_T-\delta a_T$. As in the proof of Proposition~\ref{prop:tight}, it suffices to show that $\ell_p^W(r_\delta)<\ell_p^W(\rho_T)$ with high probability, where $\ell_p^W(r)=\sup_{m\ge0}\ell_T^W(m,r)$ is concave. By Step~2 with $C=\delta$, the maximiser of $\ell_T^W(\cdot,r_\delta)$ lies in $\mu_T(1\pm\kappa)$, and the expansion of that proof gives, on this event,
\[
\begin{aligned}
\ell_p^W(r_\delta)-\ell_p^W(\rho_T)\le{}&\delta|X^W|+\delta\,\frac{2\kappa a_T}{\rho_T}\sum_{t_i\in W}x_i\\
&-\frac{\delta^2a_T^2}{2\rho_T^2}\Big(N^W-(2+4\kappa)\sum_{t_i\in W}x_i\Big),
\end{aligned}
\]
where $X^W=(a_T/\rho_T)\{\sum_{t_i\in W}u_i(\mu_T)-\rho_T\Gamma_T^W\}$ is a martingale term with $\E(X^W)^2\le a_T^2\bar\lambda_THT/\rho_T^2=H\mathcal I_T/\rho_T^2$. Moreover $a_T\sum_{t_i\in W}x_i=H\mathcal I_T+a_TM^W$ and $a_TM^W\to0$ in probability, so there is $B_H<\infty$ with $\limsup_T\PP(a_T\sum_{t_i\in W}x_i>B_H)\le\varepsilon_{\rm p}$; we work on the additional event $\{a_T\sum_{t_i\in W}x_i\le B_H\}$. On it, the middle term is at most $2\delta\kappa B_H/\rho_T$ and the corrections in the last term are at most $(2+4\kappa)a_TB_H$. Enlarging $K_1(\delta,H)$ so that $\kappa(K_1)\le\rho_T/(4\delta B_H)$, and taking $T$ large so that $(2+4\kappa)a_TB_H\le\eta H\mathcal I/2$, the middle term is at most $1$ and, on the curvature event below, the corrections are at most $a_T^2N^W/2$. For the curvature, $a_T^2N^W=\mathcal I_T\{\int_M^{M+H}V_T+(N^W-\Lambda^W)/(\bar\lambda_TT)\}$, the second term tends to zero in probability, and $\int_M^{M+H}V_T\ge\psi(V_T(M+\cdot))$ with $\psi(v)=\int_0^H\min\{v(s),1\}\,\dd s$, a bounded functional continuous for the uniform topology. By \citet[Theorem~2.6]{HorstXuZhang2023}, the shift convergence of Theorem~\ref{thm:windows} and the portmanteau theorem, for $\eta>0$,
\[
\liminf_{M}\liminf_T\PP\big(a_T^2N^W>\eta H\mathcal I\big)\ge\PP\big(\psi(V_\Lambda^{\rm stat})>\eta H\big),
\]
and by Birkhoff's ergodic theorem $\psi(V^{\rm stat}_\Lambda)/H\to c_\Lambda=\E\min\{V_\Lambda^{\rm stat}(0),1\}>0$ almost surely, so the right-hand side tends to one as $H\to\infty$ when $\eta<c_\Lambda$. Fix such $\eta$ and, given $\varepsilon_{\rm p}$, let $K_{\rm p}=2(\mathcal I/\varepsilon_{\rm p})^{1/2}$; since $\rho_T\ge1/2$ for $T$ large, $\E(X^W)^2\le4H\mathcal I$ and $\PP(|X^W|>K_{\rm p}H^{1/2})\le\varepsilon_{\rm p}$ by Chebyshev's inequality. Outside the exceptional events,
\[
\ell_p^W(r_\delta)-\ell_p^W(\rho_T)\le\delta K_{\rm p}H^{1/2}+1-\frac{\delta^2\eta H\mathcal I}{4\rho_T^2}<0
\]
whenever $\delta\ge1$, $H\ge1$ and $\delta>8(K_{\rm p}+1)/(\eta\,\mathcal I)$, a choice of $\delta$ that does not depend on $H$. Hence $\limsup^*\PP\{1-\hat\rho_T^W\ge(1+\delta)a_T\}\le2\varepsilon_{\rm p}$, which is the first claim, with $C=1+\delta$ and $\varepsilon_{\rm p}$ replaced by $2\varepsilon_{\rm p}$.

\emph{Step 4 (consistency).} On $\{1-\hat\rho_T^W\le Ca_T\}$ we have $|\hat\rho_T^W-\rho_T|\le Ca_T$, because $\hat\rho_T^W\le1$. Fix $\varepsilon_\mu>0$ and $\varepsilon_{\rm p}>0$, and let $C=C(\varepsilon_{\rm p})$ be given by Step~3. By Step~2, on the event there with $K_1(C,H)$ enlarged so that $\kappa(K_1)\le\varepsilon_\mu$, the maximiser $\hat\mu_T^W$ of $\ell_T^W(\cdot,\hat\rho_T^W)$ lies in $\mu_T(1\pm\varepsilon_\mu)$. Hence $\limsup^*\PP(|\hat\mu_T^W/\mu_T-1|>\varepsilon_\mu)\le2\varepsilon_{\rm p}$, and $\varepsilon_{\rm p}$ is arbitrary while $\varepsilon_\mu$ is fixed. \qedhere
\end{proof}

Theorem~\ref{thm:windows-nonoracle} is the version of Theorem~\ref{thm:nonoracle} that does not rely on the initial layer: with the branching ratio unknown and the record taken long after the start, the exogenous rate is learned with probability tending to one as the window grows, while the relative margin is only localised within its own order of magnitude.

\begin{remark}[Towards a self-normalised limit]\label{rem:mixed}
Simulations suggest the sharper statement $Q_T^{1/2}(\hat\mu_T/\mu_T-1)\Rightarrow N(0,1)$ (Section~\ref{sec:numerics}). The expansion in the proof reduces this to a self-normalised central limit theorem for $M_T/Q_T^{1/2}$. Random normalisation does not yield normality in general; standard results \citep{Feigin1985,JacodShiryaev2003} require in addition that $Q_T/c_T$ converge to a nondegenerate random limit for some deterministic $c_T$, jointly and stably with $M_T/c_T^{1/2}$. Since $Q_T/\mathcal I_T=\int_0^1\dd u/V_T(u)$ diverges, the relevant limit is a boundary-layer functional, and establishing it is closely tied to the rate coefficient $B_T$.
\end{remark}

\section{Perturbations of the baseline}\label{sec:baseline}

Consider now the alternative with the same mean intensity and baseline $\mu_1=\mu(1+\delta)$, so that $1-\rho_1=(1+\delta)(1-\rho)$. Then $\lambda_1-\lambda_0=(m/\rho)(\bar\lambda-\lambda_0)$, largest when $\lambda_0$ is near its floor. Let
\[
H_T=\E_0\int_0^T\varphi(\lambda_0(t),\lambda_1(t))\,\dd t,\qquad p_T(c)=\frac1T\,\E_0\int_0^T\mathbf 1\{\lambda_0(t)<c\mu_T\}\,\dd t .
\]

\begin{proposition}\label{prop:C}
For every $c\ge1$ with $c(1-\rho)<1$,
\[
H_T\ \ge\ \frac{\delta^2\,p_T(c)\,T\mu_T\,\big(1-c(1-\rho)\big)^2}{2\rho^2\,(c+\delta/\rho)} .
\]
Along~\eqref{eq:scaling}, $T\mu_T=\mu^*T^\gamma$; hence, if $\liminf_Tp_T(c)>0$, the divergence between these alternatives grows at least like $T^\gamma$.
\end{proposition}

\begin{proof}
For $x,y>0$, $\varphi(x,y)\ge(x-y)^2/(2\max(x,y))$. On $\{\lambda_0<c\mu\}$, $\lambda_1-\lambda_0\ge(m/\rho)\bar\lambda(1-c(1-\rho))$ and $\lambda_1\le\mu(c+\delta/\rho)$, using $m\bar\lambda=\delta\mu$. Integrating over the event gives the bound. \qedhere
\end{proof}

Proposition~\ref{prop:C} gives a rate, $H_T\gtrsim\delta^2p_T(c)\mu^*T^\gamma$, for two specific alternatives, conditional on the occupation of the boundary layer $\{\lambda<c\mu_T\}$. Weak convergence to a limit with an atom does not by itself control this layer, whose width in the normalised scale vanishes like $T^{-\gamma}$. The natural variable is the intensity in units of the baseline, $W_T=\lambda/\mu_T=V_T/(1-\rho_T)$.

\parhead{Boundary layers and rates.}
Weak convergence of $V_T$ to a law with an atom at zero proves that $A_T\to\infty$, but it does not identify the scale from which the atom is assembled.  The coefficient $B_T=a_TA_T=\E_{\rm time}(W_T^{-1})$ provides the relevant rate diagnostic.  A positive limit $B_T\to B_\Lambda$ corresponds to a genuine $W_T=O(1)$ boundary component and gives $T^\gamma$ information growth.  If instead $B_T=T^{-\beta+o(1)}$, the informative lower layer is effectively broader and the Fisher rate is $T^{\gamma-\beta+o(1)}$.

\parhead{Deterministic transform diagnostic.}
By the cluster representation,
$\E e^{-\theta\lambda(t)}=e^{-\theta\mu}\exp\{-\mu\int_0^t(1-\psi_\theta(r))\,\dd r\}$, where $\psi_\theta=\exp\{-\theta\rho g-\rho\,g*(1-\psi_\theta)\}$.  A positive Bernstein discretisation of $g$ turns this into a finite stiff ODE system.  We refer to the resulting calculation as a deterministic transform solver: the cluster identity is exact, while the kernel mixture is discretised.

The already frozen transform grid reveals a multiscale structure.  Besides $\beta_{\rm info}(T)=\gamma-\Delta\log A_T/\Delta\log T$, define a bulk Laplace scale $R_{1/2,T}=1/s_{1/2,T}$ by $\E e^{-s_{1/2,T}W_T}=1/2$ and its local exponent $\beta_{1/2}=\Delta\log R_{1/2,T}/\Delta\log T$.  At the largest available horizons, $(\beta_{\rm info},\beta_{1/2})$ equals approximately $(0.024,0.037)$ for $\Lambda=0.1$, $(0.042,0.109)$ for $\Lambda=0.3$, $(0.092,0.539)$ for $\Lambda=1$, $(0.163,0.595)$ for $\Lambda=2$, and $(0.394,0.597)$ for $\Lambda=4$.  Thus, for $\Lambda\gtrsim1$, the bulk is already close to the macroscopic $T^\gamma$ scale while a much thinner lower layer controls reciprocal moments.  A single scaling exponent for the whole law is therefore not supported at accessible horizons.

\begin{figure}
\centering
\includegraphics[width=.96\textwidth]{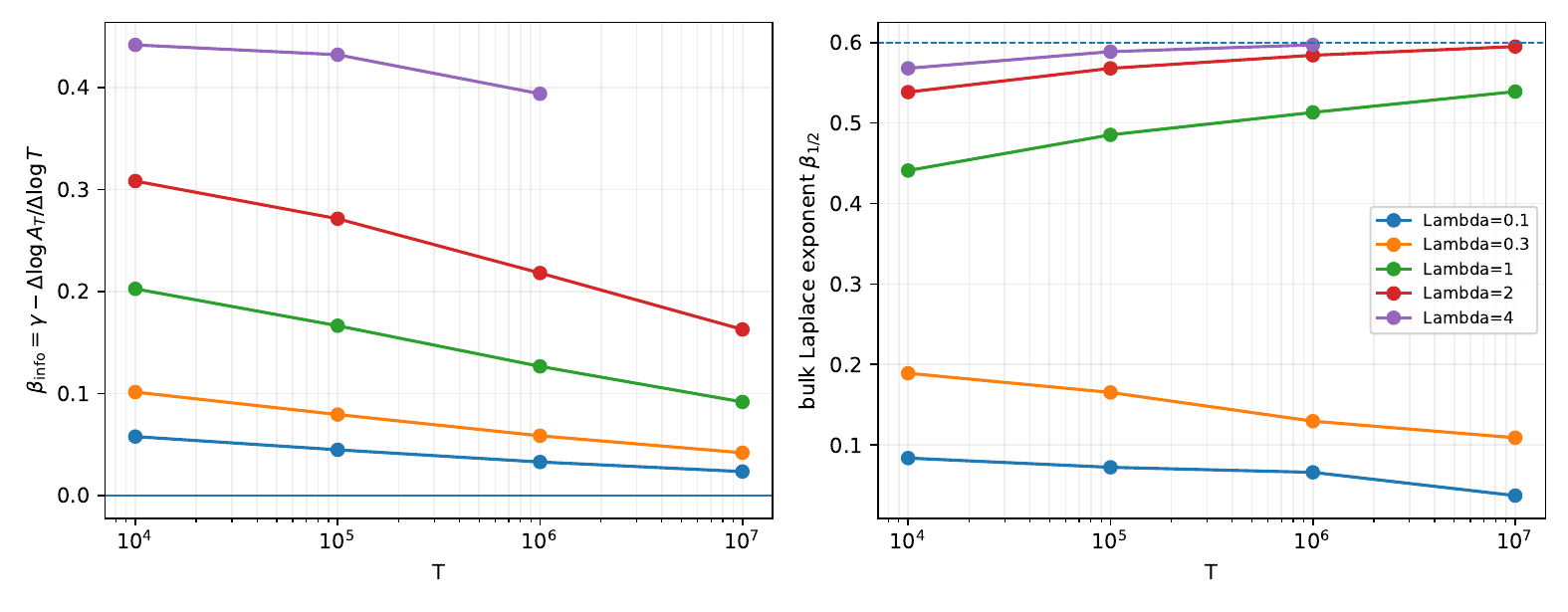}
\caption{Two boundary-layer diagnostics from the deterministic transform calculations. Left: the information-sensitive exponent $\beta_{\rm info}=\gamma-\Delta\log A_T/\Delta\log T$, which weights the lower tail through $V_T^{-1}$. Right: the bulk Laplace exponent $\beta_{1/2}$ obtained from the scale satisfying $\E e^{-W_T/R_{1/2,T}}=1/2$.  Their separation for moderate and large $\Lambda$ shows that the prelimit boundary structure is multiscale rather than described by a single layer exponent.}
\label{fig:multiscale}
\end{figure}

The earlier hypothesis of a sharp threshold in $\Lambda$ is therefore discarded.  The data instead support a crossover whose speed depends strongly on $\Lambda$.  Whether $B_T$ converges to a positive constant or decays slowly to zero remains a separate asymptotic question; either behaviour is compatible with the already established divergence $A_T\to\infty$.

\section{Numerical evidence}\label{sec:numerics}

\parhead{Saturation.} Along~\eqref{eq:scaling} with $\gamma=0.6$, oracle Kullback--Leibler divergences computed from full-history intensities stay of order $\delta^2\mu^*\lambda^*$ for perturbations of the margin at fixed baseline, for $T$ from $10^3$ to $10^5$, as predicted by Theorem~\ref{thm:A}.

\parhead{Fisher geometry.} Table~\ref{tab:fisher} evaluates Theorem~\ref{thm:fisher} by Monte Carlo, with occupation functionals computed at stratified random times ($2\times10^5$ per path). The efficient information about the margin stays below $0.05$ for $\mu^*=0.1$ and below $0.26$ for $\mu^*=1$ at every $T$, in agreement with the bound $\mathcal I_T/\rho^2$. The efficient information about the baseline grows with $T$. To separate the contribution of the initial transient, the table also reports the information on $[T/2,T]$, conditional on the past: for $\mu^*=0.1$ it keeps growing by a factor of about $3.5$ per decade, from $1.6$ at $T=10^3$ to $192$ at $T=10^7$, whereas for $\mu^*=1$ its growth weakens (from $2.2$ to $9.5$ over three decades). Doubling the number of Bernstein nodes from $160$ to $320$ changes $p_T(5)$ and $A_T$ by less than $10^{-3}$ in relative terms.

\begin{table}
\caption{Efficient Fisher information (Theorem~\ref{thm:fisher}) for the baseline and the margin on $[0,T]$ and on the trimmed window $[T/2,T]$, and boundary-layer occupation $p_T(5)$, along~\eqref{eq:scaling} with $\gamma=0.6$, $\lambda^*=1$, empty start (20 paths; 10 for $T=10^7$).}\label{tab:fisher}
\begin{tabular}{@{}cc cc cc c@{}}
\toprule
 & & \multicolumn{2}{c}{$\mathcal J^{\rm eff}_{\log\mu}$} & \multicolumn{2}{c}{$\mathcal J^{\rm eff}_{\log(1-\rho)}$} & \\
$\mu^*$ & $T$ & $[0,T]$ & $[T/2,T]$ & $[0,T]$ & $[T/2,T]$ & $p_T(5)$\\
\midrule
$0.1$ & $10^3$ & $3.7$ & $1.6$ & $0.019$ & $0.010$ & $0.77\pm0.04$\\
$0.1$ & $10^4$ & $13.7$ & $6.1$ & $0.024$ & $0.016$ & $0.74\pm0.04$\\
$0.1$ & $10^5$ & $38.5$ & $15.4$ & $0.019$ & $0.014$ & $0.59\pm0.04$\\
$0.1$ & $10^6$ & $139.6$ & $57.1$ & $0.016$ & $0.010$ & $0.56\pm0.04$\\
$0.1$ & $10^7$ & $469.4$ & $191.7$ & $0.045$ & $0.036$ & $0.48\pm0.06$\\
$1$ & $10^3$ & $9.3$ & $2.2$ & $0.215$ & $0.113$ & $0.29\pm0.04$\\
$1$ & $10^4$ & $19.5$ & $4.4$ & $0.220$ & $0.114$ & $0.13\pm0.02$\\
$1$ & $10^5$ & $47.2$ & $6.2$ & $0.229$ & $0.131$ & $0.08\pm0.01$\\
$1$ & $10^6$ & $87.3$ & $9.5$ & $0.253$ & $0.141$ & $0.03\pm0.01$\\
\bottomrule
\end{tabular}
\end{table}

\parhead{Boundary layer.} The occupation $p_T(5)$ decreases slowly for $\mu^*=0.1$ ($0.77$ to $0.49\pm0.06$ over four decades) and steadily for $\mu^*=1$ ($0.29$ to $0.03$).  The transform-based coefficient $B_T=a_TA_T$ provides a less threshold-sensitive summary: at the largest available horizon it equals $0.483,0.261,0.067,0.017$ and $0.0049$ for $\Lambda=0.1,0.3,1,2,4$, respectively.  Figure~\ref{fig:multiscale} shows that the lower tail relevant to information and the bulk of $W_T$ can evolve on very different scales.

\parhead{Windows away from the start.} Table~\ref{tab:windows} follows Theorem~\ref{thm:windows} with burn-in $M=1$, $T=10^4$ and $200$ paths per cell. As the window length $H$ grows, the probability that the window carries little exogenous information, $Q_T^W/\mathcal I_T<5$, falls from $0.29$ to $0.005$ ($\mu^*=0.3$) and from $0.81$ to $0.11$ ($\mu^*=1$), and the lower decile of $Q_T^W/\mathcal I_T$ grows roughly in proportion to $H$, as ergodic accumulation predicts. The fraction of time with $V_T<0.02$ is stable across $H$, a proxy for the stationary atom. For short windows, $5$--$10\%$ of paths stay away from zero, in line with Remark~\ref{rem:windows}.

\begin{table}[t]
\caption{Windows $[MT,(M+H)T]$ after a burn-in $M=1$, $T=10^4$, $\rho$ known, $200$ paths per cell. $\PP(\inf V_T>0.05)$ is the fraction of paths that stay away from zero over the window.}\label{tab:windows}
\centering\small\setlength{\tabcolsep}{4pt}
\begin{tabular}{@{}cc ccc c c c@{}}
\toprule
 & & \multicolumn{3}{c}{$Q_T^W/\mathcal I_T$} & & & median \\
$\mu^*$ & $H$ & $\PP(<5)$ & $q_{0.1}$ & median & time $V_T<0.02$ & $\PP(\inf V_T>0.05)$ & $|\hat\mu_T^W/\mu_T-1|$\\
\midrule
0.3 & 0.5 & $0.290$ & $1.8$ & $14.3$ & $0.251$ & $0.055$ & $0.39$\\
0.3 & 1 & $0.120$ & $3.7$ & $24.3$ & $0.220$ & $0.025$ & $0.24$\\
0.3 & 2 & $0.025$ & $9.7$ & $44.2$ & $0.172$ & $0.005$ & $0.19$\\
0.3 & 4 & $0.005$ & $25.1$ & $84.9$ & $0.166$ & $0.005$ & $0.14$\\
1.0 & 0.5 & $0.805$ & $1.2$ & $2.7$ & $0.017$ & $0.095$ & $0.38$\\
1.0 & 1 & $0.455$ & $2.6$ & $5.4$ & $0.013$ & $0.070$ & $0.23$\\
1.0 & 2 & $0.105$ & $4.8$ & $10.5$ & $0.014$ & $0.045$ & $0.21$\\
1.0 & 4 & $0.010$ & $9.5$ & $17.0$ & $0.007$ & $0.020$ & $0.16$\\
\bottomrule
\end{tabular}
\end{table}

\parhead{Joint estimation.} Table~\ref{tab:joint} reports the joint estimator of Theorem~\ref{thm:nonoracle} ($r_0=0.5$, 120 to 200 paths per cell). The median relative error of $\hat\mu_T$ decreases from $0.34$ to $0.10$ ($\mu^*=0.1$) and from $0.32$ to $0.07$ ($\mu^*=0.3$) over two decades of $T$, and $Q_T^{1/2}(\hat\mu_T/\mu_T-1)$ has standard deviation between $0.90$ and $1.20$. The estimated relative margin is poor, with medians between $2$ and $11$ times the truth and $\hat\rho_T=1$ in $12$--$28\%$ of fits, but its upper quantiles do not grow with $T$, as Proposition~\ref{prop:tight} asserts. With the likelihood restricted to the window $[T/2,T]$, conditional on the past, the median error of $\hat\mu_T$ also decreases (from $0.51$ to $0.23$ and from $0.59$ to $0.20$), more slowly: the estimator does not rely on the initial layer alone.

\begin{table}[t]
\caption{Joint estimator of $(\mu,\rho)$ with $\rho$ unknown, empty start, $\gamma=0.6$, $\lambda^*=1$, $r_0=0.5$. Columns 3--7: full record; column 8: likelihood on $[T/2,T]$ conditional on the past.}\label{tab:joint}
\centering\small
\begin{tabular}{@{}cc cc cc c c@{}}
\toprule
 & & \multicolumn{2}{c}{$|\hat\mu_T/\mu_T-1|$} & \multicolumn{2}{c}{$(1-\hat\rho_T)/(1-\rho_T)$} & & window \\
$\mu^*$ & $T$ & median & $q_{0.9}$ & median & $q_{0.9}$ & $\PP(\hat\rho_T=1)$ & median $|\hat\mu_T/\mu_T-1|$\\
\midrule
0.1 & $10^3$ & $0.34$ & $1.27$ & $6.7$ & $28.6$ & $0.14$ & $0.51$\\
0.1 & $10^4$ & $0.20$ & $0.51$ & $9.0$ & $35.1$ & $0.12$ & $0.29$\\
0.1 & $10^5$ & $0.10$ & $0.31$ & $6.6$ & $32.3$ & $0.19$ & $0.23$\\
0.3 & $10^3$ & $0.32$ & $0.94$ & $3.8$ & $14.7$ & $0.17$ & $0.59$\\
0.3 & $10^4$ & $0.18$ & $0.65$ & $3.0$ & $12.4$ & $0.21$ & $0.35$\\
0.3 & $10^5$ & $0.07$ & $0.22$ & $2.0$ & $10.8$ & $0.28$ & $0.20$\\
\bottomrule
\end{tabular}
\end{table}

\parhead{The estimator of Theorem~\ref{thm:estimation}.} With $\rho_T$ and $g$ known and empty start, we computed $\hat\mu_T$ (the root of the monotone score) and $Z_T=Q_T^{1/2}(\hat\mu_T/\mu_T-1)$ on $200$ to $300$ paths per cell (Table~\ref{tab:estimator}). The standard deviation of $Z_T$ is close to one, and the small positive bias decreases as the information grows. The law of $Q_T$ divided by its median does not change detectably over two decades of $T$ (Kolmogorov--Smirnov $p$-values between consecutive horizons up to $T=10^5$ from $0.08$ to $0.89$, and at $T=10^6$ the deciles agree with those at $T=10^5$ to within $0.03$), which indicates a nondegenerate random limit of $Q_T/c_T$; the rank correlation between $|Z_T|$ and $Q_T$ is small and decreases in absolute value as $T$ grows (for $\mu^*=0.3$ from $-0.23$ to $-0.05$; for $\mu^*=0.1$ it is $0.00$ at $T=10^6$), as asymptotic independence would require. Both features are what a mixed-normal limit needs (Remark~\ref{rem:mixed}). All cells evaluate $Q_T$ with at least one stratified point per two time units. At $\mu^*=0.1$, $T=10^6$ the standard deviation of $Z_T$ is $1.12$ (standard error about $0.05$), slightly above one, while the rank correlation between $|Z_T|$ and $Q_T$ is $0.00$.

\begin{table}[t]
\caption{Baseline estimator with $\rho_T$ and $g$ known, empty start, $\gamma=0.6$, $\lambda^*=1$: $Z_T=Q_T^{1/2}(\hat\mu_T/\mu_T-1)$ and the law of $Q_T/\mathrm{median}(Q_T)$.}\label{tab:estimator}
\centering\small
\begin{tabular}{@{}cc c cc cc c@{}}
\toprule
$\mu^*$ & $T$ & paths & mean $Z_T$ & s.d.\ $Z_T$ & $q_{0.1}$ & $q_{0.9}$ & rank corr.\ $(|Z_T|,Q_T)$\\
\midrule
0.1 & $10^3$ & 294 & $0.17$ & $1.02$ & $0.46$ & $1.40$ & $-0.09$\\
0.1 & $10^4$ & 300 & $0.17$ & $0.98$ & $0.47$ & $1.54$ & $-0.05$\\
0.1 & $10^5$ & 300 & $0.05$ & $1.00$ & $0.50$ & $1.56$ & $-0.08$\\
0.1 & $10^6$ & 300 & $0.01$ & $1.12$ & $0.47$ & $1.57$ & $0.00$\\
0.3 & $10^3$ & 300 & $0.23$ & $1.10$ & $0.45$ & $1.69$ & $-0.23$\\
0.3 & $10^4$ & 300 & $0.16$ & $1.05$ & $0.42$ & $1.70$ & $-0.16$\\
0.3 & $10^5$ & 200 & $0.10$ & $0.99$ & $0.48$ & $1.66$ & $-0.05$\\
\bottomrule
\end{tabular}
\end{table}

\parhead{Estimation and observed information.} With known kernel shape and empty start, profile-likelihood intervals for $\mu_T$ agree with the efficient-information prediction within $5$--$25\%$. A separate pathwise diagnostic of the bracket $Q_T$ of Proposition~\ref{prop:observed} (30 paths per cell) shows growth not only of its mean and median but also of its lower decile, from $1.9$ to $5.6$ and $16$ for $\mu^*=0.1$ and from $6.5$ to $8.7$ and $17$ for $\mu^*=1$ over $T=10^3,10^4,10^5$; the growth is therefore not carried by a vanishing fraction of extreme paths. For $\mu^*=0.1$ the coefficient of variation stays near $0.4$, consistent with $Q_T/\E Q_T$ converging to a nondegenerate random limit. This supports, but does not prove, a self-normalised Gaussian or likelihood limit for the exogenous direction; consistency follows from Theorems~\ref{thm:estimation}, \ref{thm:nonoracle} and~\ref{thm:windows-nonoracle}, whereas a martingale central limit theorem requires sharper control of the random bracket (Remark~\ref{rem:mixed}). Model-free estimators of $\mu_T$ based on long gaps between events fail, because with heavy-tailed memory a long gap does not imply an intensity near its floor: the informative state is the latent intensity, not an observable gap.

\parhead{Limit law.} Simulations of~\eqref{eq:limit} are consistent with the atom at zero proved by \citet{FriesenGerholdWiedermann2026}, but its mass is not reliably quantified by our methods: Euler simulations and a Riccati computation of $\E e^{-qV}$ for multifactor approximations agree for small $\Lambda$ and differ by factors of $2$ to $15$ for $\Lambda\ge1$, because finite-factor approximations smooth the singularity of $K$ at the origin that creates the atom.

\section{Discussion}\label{sec:discussion}

Near criticality with long memory, a record of length $T$ contains a diverging number of events but a bounded amount of information about the relative distance to criticality, as in local-to-unity autoregression. The exogenous rate behaves in the opposite way, and for a reason specific to roughness: the rough square-root limit spends positive time at zero, where the noise vanishes and the exogenous term alone drives the intensity. Theorems~\ref{thm:estimation} and~\ref{thm:nonoracle} show that this information is usable, under a path-dependent normalisation set by the boundary behaviour of the limit, and even when the branching ratio is unknown: the estimator only needs the margin to within its order of magnitude. Lemma~\ref{lem:layer} identifies the initial layer as a source that makes the information diverge almost surely for empty-start records, and Theorems~\ref{thm:windows} and~\ref{thm:windows-nonoracle} show that, away from the start and with the branching ratio unknown, the stationary atom takes over through ergodicity: the property belongs to the regime, not to the moment of observation.

Four questions remain open. First, the probability of learnability on a fixed window, which Theorem~\ref{thm:windows} bounds below and which we expect to be strictly less than one, and the speed at which it tends to one as the window grows. Second, the self-normalised Gaussian limit of Remark~\ref{rem:mixed}, which requires identifying the random limit of $Q_T/c_T$; simulations indicate that this limit exists, is nondegenerate and is nearly independent of the normalised score. Third, the deterministic scale of the exogenous information, governed by $B_T$ and a multiscale boundary layer. Fourth, the information in the stationary record alone, without the pre-sample history.

\section*{Funding}
No external funding was received for this work.

\section*{Declaration of competing interest}
The author declares no competing interests.

\section*{Data and code availability}
The numerical experiments use simulated data only. Code, frozen numerical
outputs, and scripts required to reproduce the reported tables and figures
are archived in Zenodo at
\url{https://doi.org/10.5281/zenodo.23035876}.

\section*{Declaration of generative AI and AI-assisted technologies in the manuscript preparation process}
During the preparation of this work, the author used OpenAI ChatGPT and Anthropic Claude for literature organisation, assistance with code development, consistency checks, and editorial revision. The author independently verified the mathematical arguments, computations, references, code, and reported results and takes full responsibility for the content of the manuscript.

\bibliographystyle{elsarticle-num-names}
\bibliography{refs}
\end{document}